\documentclass[a4paper,12pt,oneside]{article}
\usepackage[utf8]{inputenc}
\usepackage[english]{babel}
\usepackage{amsthm}
\usepackage{amssymb}
\usepackage{amsmath}
\usepackage{hyperref}
\usepackage{mathrsfs}
\usepackage{tikz}
 \usepackage{xcolor}
\usepackage[color, notref, notcite]{showkeys}     
\definecolor{refkey}{gray}{.5}   
\definecolor{labelkey}{gray}{.5} 
\usepackage{pstricks}
\usepackage[nottoc]{tocbibind}

 \usepackage{tikz}

\usepackage{graphicx}
\graphicspath{{images/}}

\usepackage{longtable,geometry}
\usepackage{color}

\usepackage{mathtools} 
\newenvironment{brsm}{
  \bigl( \begin{smallmatrix} }{%
  \end{smallmatrix} \bigr)}

\usepackage{hyperref}

\newtheoremstyle{note}{12pt}{12pt}{}{}{\bfseries}{.}{.5em}{}
\title{\LARGE\textbf{Coexistence phenomena in the Hénon family}}
\author{Michael Benedicks and Liviana Palmisano}

\makeatletter
\newtheorem{theo}[equation]{Theorem}
\newtheorem{prop}[equation]{Proposition}

\newtheorem{defin}[equation]{Definition}

\newtheorem{rem}[equation]{Remark}
\newtheorem{cor}[equation]{Corollary}

\numberwithin{equation}{section}
\newtheorem{lem}[equation]{Lemma}

\usepackage{babel}

\newcommand{\Z}{{\mathbb Z}}
\newcommand{\R}{{\mathbb R}}

\newcommand{\Cd}{{{ C}^2}}

\newcommand{\Cuno}{{{\mathcal C}^1}}

\newcommand\xqed[1]{%
  \leavevmode\unskip\penalty9999 \hbox{}\nobreak\hfill
  \quad\hbox{#1}}
\newcommand\demo{\xqed{$\square$}}

\begin{document}
\maketitle
\author
\textcolor{blue}{}\global\long\def\sbr#1{\left[#1\right] }
\textcolor{blue}{}\global\long\def\cbr#1{\left\{  #1\right\}  }
\textcolor{blue}{}\global\long\def\rbr#1{\left(#1\right)}
\textcolor{blue}{}\global\long\def\ev#1{\mathbb{E}{#1}}
\textcolor{blue}{}\global\long\def\R{\mathbb{R}}
\textcolor{blue}{}\global\long\def\E{\mathbb{E}}
\textcolor{blue}{}\global\long\def\norm#1#2#3{\Vert#1\Vert_{#2}^{#3}}
\textcolor{blue}{}\global\long\def\pr#1{\mathbb{P}\rbr{#1}}
\textcolor{blue}{}\global\long\def\qq{\mathbb{Q}}
\textcolor{blue}{}\global\long\def\aa{\mathbb{A}}
\textcolor{blue}{}\global\long\def\ind#1{1_{#1}}
\textcolor{blue}{}\global\long\def\pp{\mathbb{P}}
\textcolor{blue}{}\global\long\def\cleq{\lesssim}
\textcolor{blue}{}\global\long\def\ceq{\eqsim}
\textcolor{blue}{}\global\long\def\Var#1{\text{Var}(#1)}
\textcolor{blue}{}\global\long\def\TDD#1{{\color{red}To\, Do(#1)}}
\textcolor{blue}{}\global\long\def\Cg#1{{\color{blue} #1}}
\textcolor{blue}{}\global\long\def\dd#1{\textnormal{d}#1}
\textcolor{blue}{}\global\long\def\eqdef{:=}
\textcolor{blue}{}\global\long\def\ddp#1#2{\left\langle #1,#2\right\rangle }
\textcolor{blue}{}\global\long\def\En{\mathcal{E}_{n}}
\textcolor{blue}{}\global\long\def\Z{\mathbb{Z}}
\textcolor{blue}{{} }

\textcolor{blue}{}\global\long\def\nC#1{\newconstant{#1}}
\textcolor{blue}{}\global\long\def\C#1{\useconstant{#1}}
\textcolor{blue}{}\global\long\def\nC#1{\newconstant{#1}\text{nC}_{#1}}
\textcolor{blue}{}\global\long\def\C#1{C_{#1}}
\textcolor{blue}{}\global\long\def\meas{\mathcal{M}}
\textcolor{blue}{}\global\long\def\cSpace{\mathcal{C}}
\textcolor{blue}{}\global\long\def\pspace{\mathcal{P}}

\begin{abstract} We study the classical Hénon family $f_{a,b}:(x,y)\mapsto
  (1-ax^2+y,bx)$, $0<a<2$, $0<b<1$, and prove that given an integer
  $k\geq 1$, there is a set of parameters $E_k$ of positive
  two-dimensional Lebesgue measure so that $f_{a,b}$, for $(a,b)\in
  E_k$, has at least $k$ attractive periodic orbits and one strange
  attractor of the type studied in \cite{BC2}. A corresponding
  statement also holds for the H\'enon-like families of
  \cite{MoraViana}, and we use the techniques of  \cite{MoraViana}
  to study homoclinic unfoldings  also in the case of the original
  H\'enon maps. The final main result of the paper is the existence, within the classical Hénon
  family, of a positive Lebesgue measure set of parameters whose
  corresponding maps have two coexisting strange attractors. 

\end{abstract}

\section{Introduction}

\subsection{History}
In $1976$, the French astronomer and applied mathematician M. H\'enon made a famous computer experiment where he numerically detected but did not rigorously prove the existence of a non-trivial attractor for a two-dimensional perturbation of the one-dimensional quadratic map, $f_{a,b}:\R^2\to\R^2$ defined by 
$$
f_{a,b}\left(\begin{matrix}
x\\
y
\end{matrix}\right)=\left(\begin{matrix}
1-ax^2+y\\
bx
\end{matrix}\right)
$$
with $a=1.4$ and $b=0.3$, see \cite{Henon}. Since then, several
studies, both numerical and theoretical, have been conducted with the
aim of understanding this family of maps which is now known as
\emph{Hénon family}. The complete understanding of Hénon maps is still
quite far from being achieved. 

In his experiments Hénon also verified that attractive periodic orbits do indeed occur for
other parameter values from the same family. In view of this and of the result of S. Newhouse, \cite{Newhouse1}, stating that periodic attractors are generic, there were no reason, at the time, to eliminate the possibility that the attractor observed by Hénon was just a periodic orbit with a very high period. 

However in $1991$,  L. Carleson and the first author  proved the
existence of the attractor observed by Hénon for a positive Lebesgue
measure set of parameter values near $a=2$ and $b=0$, see \cite{BC2}. More
precisely, in the paper it was shown that if $b>0$ is small enough, then for a
positive measure set of $a$-values near $a=2$, the corresponding maps
$f_{a,b}$ exhibit a strange attractor.

\medskip
To define what we mean by a {\it strange attractor} 
we first recall that a {\it trapping region} for a map $f$ is an open
set $U$ such that

$$
\overline{f(U)}\subset U.
$$

An {\it attractor} in the sense of Conley for a map $f$ which has a trapping region   is the set
$$
\Lambda=\bigcap_{j=0}^\infty {f^j(U)}=\bigcap_{j=0}^\infty \overline{f^j(U)}.
$$

The attractor is {\it topologically transitive} if there is a
point with a dense orbit. In \cite{BC2} it was proved for a positive
two-dimensional Lebesgue measure set of  parameters  ${\mathcal
  A}$ in the $(a,b)$ space, that there is a
point $z_0(a,b)$ such that $z_1=f_{a,b}(z_0)$ satisfies the Collet-Eckmann
condition\footnote{A quadratic map $q_a(x)=1-ax^2$ satisfies the Collet-Eckman
  condition if $|(q_a^j)'(1)|\geq Ce^{\kappa j}$ for all $j\geq0$ and
  some positive constants $\kappa$ and $C$.}, i.e. that there is a constant $\kappa>0$ such that 
$$
\left|Df^n(z_1)\begin{brsm}1\\0\end{brsm}\right|\geq e^{\kappa n},\qquad\text{for all}\ n\geq 0. 
$$
It is fairly easy to see that the attractor $\Lambda$ for this
set of parameters can be identified as $\overline{W^u(\hat{z})}$,
where  $\hat{z}$ is the unique fixed point of $f_{a,b}$ in the first
quadrant, \cite{BenedicksViana}. Moreover, the fact that the
Collet-Eckmann conditions are satisfied leads to topological
transitivity, see \cite{BC2}, and the combination of
$\Lambda=\overline{W^u(\hat{z})}$ and topological transitivity
makes it appropriate to call the attractor {\em strange}.

\medskip
The techniques used in \cite{BC2} are a non trivial generalizations of
the ones presented in \cite{BC1} by the same authors for the
one-dimensional quadratic family. Those techniques  opened the way for
the understanding of a new class of non-hyperbolic dynamical systems.  

Further results have been achieved for Hénon maps by using and
developing the techniques in \cite{BC2}. In \cite{MoraViana} the
results of \cite{BC2} are obtained for a general perturbation of the
family of quadratic maps on the real line, called Hénon-like
family. The  statistical properties, the existence of a
Sinai-Ruelle-Bowen (SRB) measure, exponential decay of correlation and
a central limit theorem were
studied in \cite{BenedicksYoung1} and \cite{BenedicksYoung2}.
Furthermore the metric properties of the basin of attraction of the strange attractor
was studied in \cite{BenedicksViana}. In that paper it was proven that
Lebesgue almost all points in the topological basin for the attractor
$$
B=\bigcup_{j=0}^\infty f^{-j}(U),
$$
are generic for the SRB measure. Here $U$ is the trapping region as above.

Other more recent approaches to generalizations of this class of
dissipative attractors were given by Wang and Young in
\cite{WangYoung1}, \cite{WangYoung2} and by Berger in \cite{Berger1}.

In the present paper in Theorem \ref{B}, we show that coexistence of periodic attractors
and strange attractors occur in the H\'enon family for a positive
Lebesgue measure set of parameters. Our proof is mainly based on the
techniques in \cite{BC2}. However the construction of the periodic
attractors is inspired by \cite{Th}, where H. Thunberg proved the
existence of attractive periodic orbits for one-dimensional quadratic
maps for parameters that accumulate on the ones corresponding to the
quadratic maps with absolutely continuous invariants measures of
\cite{BC1} and \cite{BC2}. A result similar to that of \cite{Th} has
been obtained for H\'enon maps in \cite{Ures}.

After the completion of this paper it was pointed out to the
authors by P.~Berger that there is an alternative approach to Theorem \ref{A}
using the method of Newhouse \cite{Newhouse1},\cite{Newhouse2}, in
particular the version of the Newhouse theory for one-dimensional
families of map presented in \cite{Ro}. The present approach is however
gives a different, more constructive, approach to the phenomena of
Newhouse. In particular Baire Category arguments are avoided. 

Furthermore this constructive method allows us to  we prove the existence of a
positive two-dimensional Lebesgue measure set of parameters in the
H\'enon family for which there exist two coexisting strange
attractors. This result is stated as Theorem \ref{D} and is the main
result of the present paper,


\medskip
The next section contains more details about our main results.

\subsection{Statement of the results}

We now present our main results. We first give the definition of
H\'enon-like families as in \cite{MoraViana}.

\begin{defin}\label{henonlike} An $a$-dependent
  one-dimensional parameter family of maps $F_a$ is called a  {\it
    H\'enon-like} family 
if

$$
F_a(x,y;b)=\begin{pmatrix}1-ax^2\\0
\end{pmatrix}+\psi(a,x,y;b),
$$
and we have the following properties:

\begin{itemize}
\item[(i)]
  $\psi$ satisfies the condition
  $$
  ||\psi||_{C^3}\leq K b^t.
  $$
\item[(ii)]  
Let  $A,B,C,\, D$ be the matrix element of 
$$
DF_a=\begin{pmatrix}A&B\\C&D
  \end{pmatrix},
$$  
 and assume  $A$, $B$, $C$, $D$, satisfies the conditions stated in
 Theorem 2.1 of \cite{MoraViana},
\item[(a)] $|A|\leq K, \sqrt{b}/K\leq |B|\leq K  \sqrt{b},
  \sqrt{b}/K\leq |C|\leq K \sqrt{b},\ b/K\leq |\det\ DF_a |\leq K b$,
  $||DF_aa||\leq K$ and $||DF_a^{-1}a||\leq K/b$.
\item[(b)] $||D_{(a,x,y)}A||\leq K$, $||D_{(a,x,y)}B||\leq
  K^{1/2+t}$, $||D_{(a,x,y)}C||\leq K^{1/2+t}$, $||D_{(a,x,y)}D||\leq
  K^{1+2t}$. Moreover $||D_{(a,x,y)}(\det DF_a)||\leq Kb^{1+t}$ and
  $||D^2F_a||\leq K$.
\item[(c)] $||D^2_{(a,x,y)}A||\leq Kb^t$, $||D^2_{(a,x,y)}B||\leq
   Kb^{1/2+2t}$, $||D^2_{(a,x,y)}C||\leq
   Kb^{1/2+2t}$, $||D^2_{(a,x,y)}D||\leq 
   Kb^{1+3t}$. Finally   $||D^2_{(a,x,y)}(\det DF_a)||\leq
   Kb^{1+2t}$ and $||D^3F_a||\leq Kb^t$.
  \end{itemize}
\end{defin}
\begin{rem}
  The original H\'enon family corresponds to
  $$
\varphi(x,y;b)=\sqrt{b}\begin{pmatrix}y\\x
\end{pmatrix}.
  $$
  
\end{rem}  

\begin{theo}\label{A} Suppose $F_a(.,.;b)$ is an $a$-dependent H\'enon-like family as
  in Definition \ref{henonlike}. Then there is a $b_0>0$ so that for
  all $k\geq 1$, and all $0<b<b_0$,  there is a set of $a$-parameters
  $A_{k,b}$ ( with fixed $b$) which has positive one-dimensional
  Lebesgue measure, i.e. $|A_{k,b}|>0$ and such that for all $a\in A_{k,b}$, $F_a(.,.;b)$  has at least $k$ attractive
   periodic orbits and at least one strange attractor of the type constructed
   in \cite{BC2} and \cite{MoraViana}.
  \end{theo}
The method introduced to prove Theorem \ref{A} gives also the following result.
\begin{theo}\label{B} Suppose $F_a(.,.;b)$ is a H\'enon-like family as
  in Definition \ref{henonlike}. If $b_0>0$ is sufficiently small, then for all 
  $0<b<b_0$ and for all $a$ in some set $A_{\infty,b}$, $F_a(.,.;b)$ has infinitely many coexisting attractive periodic orbits (the
  Newhouse phenomenon). 
  \end{theo}
Theorem \ref{A} and Theorem \ref{B} hold for the
original H\'enon family.
 \begin{theo}\label{C} 

        Consider the original H\'enon family $f_{a,b}$, $0<a<2$, $0<b<1$.

   \begin{itemize}
\item[(a)] There is a set of positive two-dimensional Lebesgue measure
  of parameters with at least $k\geq1$ attractive periodic orbits and
  one H\'enon-like strange attractor.
\item[(b)] There are parameters in the H\'enon family for
  which there are infinitely many attractive periodic orbits.
  \end{itemize}
 \end{theo}  
 
The existence of H\'enon and H\'enon-like maps in one-parameter
families with infinitely many sinks has already been established in
\cite{Ro}, \cite{GST} and \cite{GS}.
In difference to the previous approaches, the present methods of proof are
completely constructive. In particular, the methods avoid Baire category
arguments, the Newhouse thickness criterium  and the persistance of
tangencies is not used.

Our method allows also to obtain a stronger result about the coexistence of two chaotic, non-periodic attractors. The following can be considered as the main theorem of the paper.
\begin{theo} \label{D}There is a positive two-dimensional Lebesgue
  measure set of parameters ${\mathcal A}$, such that for $(a,b)\in{\mathcal A} $, the maps of the
  H\'enon family $f_{a,b}$ have two coexisting strange attractors.
\end{theo} 

Our results can be viewed as some steps in the Palis program, see \cite{Palis}, aiming to describe coexistence phenomena for dissipative surface maps. Other coexistence results has been obtained in e.g. \cite{BenedicksMartensPalmisano, Berger2, Palmisano}.

\paragraph{Acknowledgements.} The first author was supported by the
Swedish Research Council Grant 2016-05482. The
second author was supported by the Trygger Foundation, Project CTS
17:50 and the research was partially summorted by the NSF grant
1600554 and the IMS at Stony Brook University. The authors would like
to thank P. Berger, L. Carleson and J-P Eckmann for helpful
discussions. The project was initiated at Institute Mittag Leffler
during the program Fractal Geometry and Dynamics, September 04 --
December 15, 2017.

\bigskip

\section{Overview of results and methods on H\'enon and H\'enon-like maps}
In this section we collect definitions and constructions by \cite{BC2} and \cite{MoraViana} which will be used in the sequel.
We briefly review the construction of Collet-Eckmann maps in the
quadratic family and the H\'enon family of \cite{BC1}, \cite{BC2}, and
the corresponding construction in \cite{MoraViana}. For more details we refer to
the original papers.
\subsection{The one-dimensional case}\label{onedimensional-case}
\medskip
Let us first consider the quadratic family $q_a(x)=1-ax^2$  and we write
$\xi_j(a)=q_a^j(0)$, $j\geq 0$. We start with an interval $\omega_0=[a',a'']\subset
(0,2)$ and  very close to 2. We partition
$(-\delta,\delta)=\bigcup_{|r|\geq r_\delta} I_r$, where $I_r=(e^{-r},e^{-r+1})$,
$I_r=-I_r$ and $I_r=\bigcup_{\ell=0}^{r^2-1} I_{r,\ell}$, where the
intervals $I_{r,\ell}$ are disjoint and of equal length. The
definition is similar for negative $r$:s. We do an explicit preliminary
construction of the first {\it free return}  so that it satisfies
$$
\xi_{n_1}(\omega)=I_{r_\delta,\ell},
$$
i.e. a parameter interval $\omega$ is mapped by the parameter dynamics $a\mapsto
\xi_{n_1}(a)$ to a  parameter interval in the partition
$\{I_{r,\ell}\}$. Here $r$ is chosen so that $e^{-r}\geq e^{-\alpha
  n(\omega)}$, and therefore Assertion 4, (ii), in Subsection \ref{ss:twod} is satisfied. This
condition is called the basic assumption (BA) in \cite{BC2}.

We give a brief description of the constructions in \cite{BC1}, \cite{BC2}.
At the $n$:th stage of the construction, we have a partition
${\mathcal P_n}$ and for $\omega\in{\mathcal P}_n$, when $n=n_k$ is  a free return,
we have
$$
\xi_n(\omega)\subset I_r\cup I_{r-1}\qquad\text{if}\ r>0.
$$
(The case $r<0$ is analogous.)
We define the bound period at a free return as the maximum integer $p$
so that

\begin{equation}
|\xi_{n+j}(a)-\xi_j(a')|\leq e^{-\beta j} \qquad \forall a,a'\in\omega,\
\forall j\leq p.\label{bound-period}
\end{equation}

After the bound period there is a {\it free period} of length $L$, during which
the corresponding iterates are called free, and at time $n+p+L$ we
have a return, at which
$$
\xi_{n+p+L}(\omega)\cap (-\delta,\delta)\neq\emptyset.
$$
This corresponds to a new free return to an interval $I_r$, which can either
be {\it essential}, i.e. the image covers a whole
$I_{r,\ell}$-interval or it is contained in the union of two adjacent
such intervals. The latter case is called an {\it inessential} free
return.
If we have an essential return the part  of $\omega\in {\mathcal P}_{n-1}$, which is mapped to $(-e^{-\alpha n},e^{-\alpha
  n})$ is deleted and we define the partition ${\mathcal P}_n$ by
pulling back the intervals $\{I_{r,\ell}\}$ to the parts of $\omega$
that remain after deletions. The union of the partition elements of
the parameter space that
remain at time $k$ is written as  $A_k=\bigcup_{\omega\in{\mathcal
    P}_k}$. The numbers $\alpha$ and $\beta$ are
small and positive. In the one-dimensional case one can choose
$\alpha=\frac1{400}$ and $\beta=\frac{1}{100}$. 
Define $\rho_k=|r_k|$, $k=0,\dots,r_s$.  Then
$(\rho_0,\dots,\rho_s)$ is an itinerary, which essentially determine
the derivative expansion that from  free return time $n_k$ to free
time $n_{k+1}$ is always
\begin{equation}\label{eq:largedeviation}
\geq \frac{e^{-3\beta \rho_k}}{e^{- \rho_k}}.
\end{equation}

A combinatorial argument shows, see Section 2.2 in \cite{BC2}, that
there are {\it escape situations} for partition elements $\omega$  at
times $\tilde{E}(\omega)$. The definition of an escape situation is
somewhat arbitrary but let us define it as a pair
$(\omega,\tilde{E})$, $\omega\in {\mathcal P}_{\tilde{E}}$ which is
defined so that $\omega$, under the parameter dynamics, is mapped
to an interval of size $\geq\frac{1}{10}$ at time $\tilde{E}$.

The
escape time $\tilde{E}$ has a distribution depending essentially on the 
itineraries $(\rho_0,\rho_1,\dots,\rho_s)$ of the subintervals 
of $\omega\in{\mathcal P}_{n_0}$.
By Section 2.2 of \cite{BC2} we have
\begin{itemize}
	\item the total time $T$ spent in an itinerary $(\rho_0,\rho_1,\dots,\rho_s)$ satisfies
	$$
	T\sim \sum_{j=0}^{s} \rho_j
	$$
	\item  $z_{n_i}$, at the return times $n_i$, $i=1,2,\dots,s$, can be viewed as almost independent random variable,
	\item the distribution of the escape times after the parameter selection satisfies
	$$
	\left|\left\{a\in\omega_0\left|\right. E(a)>t\right\}\right|\leq C	\left|\omega_0\right|e^{-\gamma t}	
	$$
	with $\gamma,C>0$.
\end{itemize}
This is known as the large deviation argument.
\subsection{The two-dimensional case}\label{ss:twod}

By perturbing the quadratic family  interpreted as an endomorphism
$(x,y)\mapsto (1-ax^2,0)$, where $a$ is
close to 2, we obtain a  H\'enon-like map of the type given in Definition
\ref{henonlike}.

If the map is
orientation reversing it has a fixed
point $\hat{z}\approx (\frac12,0)$ in the first quadrant. For small $b$, the unstable eigenvalue
$\lambda_u$ is approximately equal to $-2$ and the product of the
stable and unstable eigenvalues $\lambda_u$ and $\lambda_s$, i.e.
$\lambda_u\cdot\lambda_s =\hat{d}$, where $\hat{d}=\det(DF_a(\hat{z}))$.

 One of the main new
ingredients in the two-dimensional theory is that the critical point 0
of the one-dimensional map in the $n$:th stage of the induction is replaced by a critical
set ${\mathcal C}_g$, $g\leq Cn/\log(1/b)$. There is also a special set of critical points $\Gamma_N \subset {\mathcal C}_g$
on which the induction is carried on, and which is increased as the
induction index $n$ grows. (Note that the critical set $\Gamma_N$ in
the construction is only changed for a special sequence $\{N_k\}$ of
times $n$. The induction on $n$ is done for $n$ satisfying $N_k\leq n
\leq N_{k+1}$.) In the case of H\'enon-like maps it is most
natural to define instead of the critical point, the critical value.
The unstable manifold
$W^u(\hat{z})$ of
the fixed point has a sharp turn close to $x=1$. The critical
value $z_1$ has the property that there is $\kappa>0$ so that 
\begin{equation}\label{CE-equation}
\left|DF^j(z_1)\begin{brsm}1\\0\end{brsm}\right|\geq e^{\kappa j}\qquad \text{\rm for all}\quad 0\leq j\leq n.  
\end{equation}
The first approximation of $z_1$ is defined as the tangency point
between the vector field defined by the most contracting direction of $DF(z)$ close to $(1,0)$. Successively the equation
\eqref{CE-equation} is verified by induction for higher and higher $n$
and this allows most contracting directions of higher orders to be
defined. This makes better and better approximations of the critical value. 
This allows us to define the image $z_2$ of the critical value $z_1$
under the maps $F$, and also the critical point $z_0$ as $z_0=F^{-1}(z_1)$. 
The critical point $z_0$ will play a crucial role in our
construction. Note that all this is defined for an interval $\omega\in{\mathcal
  P}_n$ and all points $a$ of $\omega$ have equivalent $z_0$, $z_1$
and $z_2$. An arbitrary point $a\in\omega$ can be used for the
definitions.

We now define for $a\in\omega$ the first generation $G_1$ of $W^u(\hat{z})$
as the segment of $W^u(\hat{z})$ from $z_1$ to $z_2$. We also make the
notation $W_1=G_1$ and inductively define $W_{k+1}=F_a(W_k)$ and then
$G_k=W_{k+1}\setminus W_k$ for $k\geq 1$.

The induction proceeds by using information of the  critical points
$\Gamma_N$  (and
corresponding critical values) defined on segments of $W^u(\hat{z})$ of generation $\leq g=CN/\log(1/b)$, where $C$
is a numerical constant. One can consider $\Gamma_N$ as the set of
``precritical points''. A succesive modification procedure at the times $N_k$
will make the
``precritical points'' converge to the final critical points. 

\medskip
We require the following: 

\smallskip
Consider a free return time $n$ of the induction, and for all
$\omega\in {\mathcal   P}_n$ all critical values $z_1$ associated with
$\Gamma_N$ satisfy  

\bigskip
{\it Assertion 4 of \cite{BC2},} equation (12b), p.42. in \cite{MoraViana}

There is a constant $\kappa>0$ so that 
\begin{itemize}
\item[(i)] $\left|DF_a^j(z_1) \begin{brsm}1\\0\end{brsm}\right|\geq e^{\kappa j}\qquad\forall j \leq n $;
  
\item[(ii)] $\text{dist}_h(F_a^j(z_1),\Gamma_N)\geq e^{-\alpha
    j}\qquad\forall j\leq n$. 
 \end{itemize} 
 The formal definition of $\text{dist}_h(F_a^i(z_0),\Gamma_N)$,
 denoted by $d_i$ in \cite{BC2}, is given in Assertion 1, p. 127,
 in that paper and this quantity at returns satisfies
 $$
3|z_i-\tilde{z}_0^{(i)}|\leq  d_i(z_0)\leq 5|z_i-\tilde{z}_0^{(i)}|,
$$
where $z_i$ is at returns, by construction located horizontally to its
{\it binding point}  $\tilde{z}_0^{(i)}\in \Gamma_N$.
The condition (ii) is called the Basic Assumption (BA) in \cite{BC1},
\cite{BC2}. Roughly speaking, a binding point is chosen at a suitable
horizontal location so that the splitting argument, and   the bound
period distorsion estimates of the corresponding $w_\nu^*$-vectors
will be valid, see Subsection \ref{ss:splitting} below.


\subsection{Splitting algorithm}\label{ss:splitting}

Now we recall  the splitting algorithm for expanded vectors as in \cite{BC2},
  and \cite{MoraViana} p. 40-41. Let $w_\nu=DF^\nu(z_0)\begin{brsm}1\\0\end{brsm}$,
  and we write
  $$
  w_\nu=E_\nu+w_\nu^*.
  $$
 $E_\nu$ corresponds to the part of $w_\nu$ that is in a folding
 situation, i.e. there are various terms in $E_\nu$ that come from a
 splitting at a previous return. In particular if $\nu$ is outside of
 all bound periods $w_\nu=w_\nu^*$.

 \medskip
 We now summarize an essential part of Assertion 4  concerning distorsion
 of the vectors $w_\nu^*$ during the bound period, which has  an
 analogous definition to that in the one-dimensional case given in \eqref{bound-period}.

\medskip
There are constants $C_0$ and $C$, such that for all critical
points $z_0\in\Gamma_N$
 
\begin{itemize}
 
\item[(a)] If $p$ is the binding time  for $\zeta_0$ to $z_0$
$$
C^{-1}\leq \frac{||w_\nu^*(\zeta_0)||}{||w_\nu^*(z_0)||}\leq C, \qquad
0\leq \nu\leq p.
$$

\item[(b)] Let  $z_0\in\Gamma_N$, let $\zeta_0$ and $\zeta_0'$ be
  two points bound to $z_0$ during time $[0,p]$ and let $n$ be the
  first free return $n\geq p$. Furthermore let $w_\nu^*(\zeta_0)$ and
  $w_\nu^*(\zeta_0')$ be the associated vectors of the splitting
  algorithm. We write the vectors in polar coordinates, where  $M_\nu(\cdot)$
  denotes the absolute value and $\theta_\nu(\cdot)$ the argument, and
  measure the  distance between the orbits using

  $$
  \Delta_i(\zeta_0,\zeta_0')=\max_{0\leq j\leq i}|\zeta_j-\zeta_j'|.
  $$
  Then there is a constant $C_0$ such that, if
  $$
  \sum_{j=1}^k \frac{\Delta_j}{d_j(z_0)}\leq \frac{1}{C_0},\qquad \text
  {and}\ k\leq \min(n,N),
  $$
  then if $\nu\leq k$
  \begin{equation}\label{eq:modulus}
  \frac{M_\nu(\zeta_0)}{M_\nu(\zeta_0')}\leq
  \exp\left\{C_0\sum_{j=1}^\nu \frac{\Delta_j}{d_j(z_0)}\right\},
  \end{equation}
  and

 \begin{equation}\label{eq:angle2}
  |\theta_\nu(\zeta_0)-\theta_\nu(\zeta_0')|\leq 2 b^{1/4}\Delta_\nu.
 \end{equation} 
\end{itemize}

Very similar estimates appear in Lemma 10.2, in \cite{MoraViana}.
Their estimate in the Modulus equation \eqref{eq:modulus} is better with the quantity
$$
\Theta_k=\Theta_k(\zeta_0,\zeta_0')=\sum_{s=1}^\nu b^{(s-\nu)/4} |\zeta_s-\zeta_s'|,
$$
instead of $\Delta_i(\zeta_0,\zeta_0')=\max_{0\leq j\leq i}|\zeta_j-\zeta_j'|$.

We have written \eqref{eq:angle2} with the constant $2b^{1/4}$ as in
\cite{MoraViana} instead of $2b^{1/2}$ as in \cite{BC2}  since our
estimates are required to work also in the more general setting of
H\'enon-like maps.


\subsection{Derivative estimates and $C^2(b)$ curves for H\'enon-like maps}
We also need at several places that uniform expansion of the $x$-derivative 
of the $n$:th iteration of a function $F(x;a)$ automatically gives a uniform comparasion of $a$ and
$x$-derivatives of the iterated function. In the one-dimensonal case
this is  formulated abstractly in Lemma 2.1 in \cite{BC2}. The
corresponding estimate in the two-dimensional case is \cite{BC2} lemmas
8.1 and 8.4 and \cite{MoraViana} Lemma 11.3, which we formulate as
a distorsion result for the $w_\nu^*$ vectors of the splitting algoritm.


\begin{lem}\label{par-phase-dist1}
We consider the critical orbit $z_\nu(a)$ as a function
of the parameter $a$. We denote its derivative with respect to $a$ by
$\dot{z}_\nu(a)$. Then the following holds

\medskip
For all $2\leq\nu\leq n$ and $a\in{\mathcal
  P}_{\nu-1}(\omega)\subset E_{\nu-1}(z_0)$ we have
\begin{itemize}

\item[(i)]  
$$
\frac{1}{100}\leq \frac{||\dot{z}_\nu(a)||}{||w_\nu^*(a)||}\leq 100.
$$
Moreover if $\nu$ is a free iterate then
\item[(ii)]
$|\text{\rm
  angle}(\dot{z}_\nu(a),w_\nu^*)|\leq b^{t/2}$.
\end{itemize}
\end{lem}
We also need a statement about distorsion for the tangent vectors of
the parameter dependent curves $a\mapsto z_\nu(a)$, which can be formulated as follows.

\begin{cor}\label{cor:pardist} There is a constant
  $C(K,\alpha,\beta,\delta)$, so that  if
  $\nu$ is a free return then if $\omega\in{\mathcal P}_{\nu-1}(z_0)$
  then for all $a,a'\in\omega$
$$
  \frac{||\dot{z}_\nu(a')||}{||\dot{z}_\nu(a)||}\leq C
\quad \text{\rm and}\quad \text{\rm angle}(\dot{z}_\nu(a'),\dot{z}_\nu(a))\leq 10 b^{1/4}.
$$
\end{cor}  

For the construction of two strange attractors, Theorem \ref{C}, we
also need the distorsion control of the $b$-derivatives given in Lemma
\ref{pardist} below.

In several places, in particular for parameter dependent curves and
pieces of unstable manifolds,  it is relevant that the corresponding
curves segments are $C^2(b)$-curves which in the setting of the
H\'enon-like maps of \cite{MoraViana}, has the following definition.

\begin{defin}\label{C2bcurve}
A curve $\gamma(x)=(x,h(x))$, $x_1\leq x\leq x_2$ is called a
$C^2(b)$-curve if the curve is $C^2$, and there is a constant $C$ so
that $|h'(x)|\leq Cb^t$ and  $|h''(x)|\leq Cb^t$ for $x_1\leq x \leq x_2$. The constant $t>0$
appears in the definition of the H\'enon-like maps.
  
\end{defin}  

\subsection{Stable and unstable manifold}

We also need some geometric information on the attractor.  A reference
is \cite{MoraViana}, Section 4,  but we will also need  two 
quantitative statements on the stable and unstable manifolds of the
fixed point formulated in lemmas \ref{stablemanifold}, \ref{unstable} and \ref{lem:equidistsm} below.

\begin{lem}\label{stablemanifold}
Let $\gamma^s_a$, $a\in\tilde\omega_0$, be the first leg of the stable
manifold of $\hat z(a)$ pointing in the negative $y$ direction. Then
$\gamma^s_a$ at all points has slope bounded below by $K/\sqrt{b}$
where $K$ is a numerical constant. Moreover $\gamma^s$ has a $C^1$
dependence on $a$. Also the downwards pointing leg $\gamma^s_a$ of
$W^s(\hat{z})$ intersects $W^u(\hat{z})$ at a homoclinic point $\hat{z}'$.
\end{lem}
\begin{proof} We consider the orientation reversing case when the fix
  point $(\hat{x},\hat{y})$ satisfies $\hat{y}>0$.

  By the  $C^1$-version of the stable manifold theorem,
  there is a small segment of the $\gamma_s$-leg pointing down. Note
  that we do not have control of the size of this leg. It depends on
  $a_0$, the middle point of $\tilde{\omega}_0$, and $b$. By $C^1$ continuity of the stable manifold we can
  choose a sufficiently small segment $\Gamma_0$ so that its slope is close to
  the slope at the fixed point. As in \cite{MoraViana} the derivative of the
  map is defined as
  $$
  DF_a(x,y)=\left(\begin{matrix}
    A& B\\
    C& D
    \end{matrix}\right)(a,x,y).$$
  The stable direction at the fixed point has approximate slope $s_0$,
  where
  $$s_0=\frac{-2a\hat{x}}{B},
  $$
  and by continuity this is true also for points of $\Gamma_0$. Now
  define inductively
  $\Gamma_{n+1}=F_a^{-1}(\Gamma_n)$ for $n \leq n_0$, where $n_0$ is
  determined so that $(x,y)\in\Gamma_n$ for $n\leq n_0$ should satisfy $y\geq
  \frac{7}{8} \hat{y}$. Note that we have strong expansion of the
  inverse map $F_a^{-1}$ and $n_0$ is finite.

  Next we verify that the cone defined by
  $$
  |s-s_0|\leq\frac{1}{10}|s_0|
  $$
  is invariant under $DF_a^{-1}$. For this we use the derivative
  estimates of $A$, $B$, $C$, $D$ and the determinant $AD-BC$ in
  \cite{MoraViana}, Theorem 2.1. This will hold for the  sequence of curve segments
  $\{\Gamma_n\}$, $n\leq n_0$. The length of $\Gamma_{n_0}$, will
  be greater or equal to $\frac{1}{8}\hat{y}>0$.
We now do two final iterates and conclude that $\Gamma_{n_0+2}$ has
a subcurve with vertical slope $\geq K/\sqrt{b}$ and length $\geq
C\hat{y}b^{-1}$. It follows that we have the required homoclinic
intersection $\hat{z}'$, compare Lemma \ref{homoclinic}. \end{proof}

\begin{lem}\label{unstable} Consider a family of  H\'enon-like maps
  $F_{a}(.,.;b)$ which is area reversing.  Let a time $\nu$ be given and let
  a parameter interval of $a$-values, $\omega\in{\mathcal P_\nu}$. For
   $a\in\omega$   there is a critical point $z_0$ and a critical
  orbit $z_1$, $z_2$, $z_3$ located on $W^u(\hat{z})$. Let $\gamma_u$
  be the segment of  $W^u(\hat{z})$ from $z_2$ to $z_3$. Then for a
  suitable choice of $\delta_0$, the
  curve segment
  $$
  \gamma^u_1=\gamma_u\cap \{(x,y): x\geq -1 + \delta_0\}
  $$
  is an approximate parabola and the two segments
  $$
\gamma^u_1\cap \{(x,y): x\leq  1-\delta_0\}
  $$
  are two $C^2(b)$ curves.
\end{lem}  
\medskip

{\it Sketch of proof.} For the first part of the proof we follow
\cite{MoraViana}, Section 7. In formula (2), p.30, they state that  
the unstable manifold restricted to $G_0\cap \{|x|\leq 1-\delta_0\}$ can
  be viewed as the graph $y(x)=y_\varphi(a,x)$ with
  $$
||y_\varphi||_{C^2}\leq \text{const}\, b^t,
  $$
If we iterate the unstable manifold once it follows that it folds to a
parabola. From a curvature argument, see \cite{MoraViana} Lemma 9.3, it
follows that the curve is $C^2(b)$.\demo

We will later need information on the structure of the stable
manifold of the fixed point $\hat{z}$.

\begin{lem}\label{lem:equidistsm} There is an approximate equidistribution of pieces of the stable
manifold $W^s(\hat{z})$, with a definite slope $s$, $|s|\geq
\text{\rm Const.}\ \delta$  that intersect $\{(x,y): |x|\geq
\delta\}$. The interspacing of the the legs of $W^s(\hat{z})$ is $\sim
  \frac{\pi}{2}\cdot\frac{1}{3\cdot 2^k}$.

\end{lem}

\begin{proof}
Consider the tent map $\xi\mapsto 1-2|\xi|$. It has a
  fixed point $\xi=\frac13$. The preimages of this fixed point are
  located at

  $$
  \xi_{\nu,k}=\frac{\nu}{3\cdot 2^k},\qquad \nu=-3\cdot 2^k+1,\dots,3\cdot 2^k-1
  $$
  The corresponding points for the quadratic map $x\mapsto 1-2x^2$ are
  given by $x_{\nu,k}=\sin\frac{\pi}{2}\xi_{\nu,k}$. This means that
  the interspacing of the legs of $W^s(\hat{z})$ is as required.
\end{proof}

\subsection{The Stable Foliation and its properties}\label{stablefoliation}

The stable foliation of order $n$ for different values of $n$ will play
an important role in the following, in particular in the capturing argument
in Section \ref{sec:capturing} and in the construction of the sink in
Section \ref{sink}. This construction of the stable foliation appears in \cite{BC2}, but
we will use the version in \cite{MoraViana}, Section 6.

We will need some lemmas about the expansion properties of the
maps. Because of the dissipative properties of the maps these will lead
also to the existence of  contractive vector fields and a
corresponding stable foliation. 

Let $F$ be a H\'enon-like map and denote by
$M^{\nu}(z)=DF^{\nu}(z)$.  Let $u_0$ be a
tangent vector of $W^{u}(\hat z)$ near $\hat z$. Let
$\zeta_0=(\xi_0,\eta_0)$ be a point on the unstable manifold,
satisfying  $|\xi_0|\geq\delta$ and for any $1\leq\nu\leq n$,
$\left\|M^{\nu}(\zeta_0)u_0\right\|\geq\kappa^{\nu}$. We get an
expansive behaviour of horizontal vectors, compare Corollary 6.2 in  \cite{MoraViana}.
Here $\kappa<1$ is allowed. We need a condition similar to partial
hyperbolicity relating  $b$ and $\kappa$ such as
$\sqrt{b}\leq\left(\kappa/{10 K^2}\right)^4$, compare 
 the hypothesis of Lemma \ref{lemma6.4} below.

\begin{lem} \label{lemma6.2} Assume that $\zeta_0=(\xi_0,\eta_0)$ is a point on the
  unstable manifold satisfying $|\xi_0|\geq \delta$ and

  \begin{equation}\label{eq:expansion}
\left\|M^{\nu}(\zeta_0)u_0\right\|\geq\kappa^{\nu},\qquad 1\leq\nu\leq n.
\end{equation}

Then all $1\leq\nu\leq n$ and for all unit vector $v_0$ with
$|\text{\rm slope}(v_0)|\leq\frac{1}{10}$, 
$$
\left\|M^{\nu}(\zeta_0)v_0\right\|\geq \frac{1}{2}\left\|M^{\nu}(\zeta_0)\right\|.
$$  
\end{lem}

We will also need Lemma $6.3$ in \cite{MoraViana} which implies  
estimates  of the norms and angles of the expanded vectors.
\begin{lem}\label{lemma6.3}
Let $\zeta'_0$ and norm $1$ vectors $u,v$ satisfying
$$|\zeta_0-\zeta'_0|\leq\sigma^n\text{ and } \left\|u-v\right\|\leq\sigma^n $$ with $\sigma\leq\left(\frac{\kappa}{10 K^2}\right)^2$, then 
\begin{itemize}
\item[(a)]$\frac{1}{2}\leq\frac{\left\|M^{\nu}(\zeta_0)u\right\|}{\left\|M^{\nu}(\zeta'_0)v\right\|}\leq 2$,
\item[(b)]$\left|\text{\rm angle}\left(M^{\nu}(\zeta_0)u, M^{\nu}(\zeta'_0)v\right)\right|\leq\left(\sqrt{\sigma}\right)^{2n-\nu}\leq \left(\sqrt{\sigma}\right)^{n} $.
\end{itemize}
\end{lem}

Observe that, by Lemma \ref{lemma6.2}, the conclusions of Lemma \ref{lemma6.3} are verified for all unit vectors $u,v$ such that $\left\|u-v\right\|\leq\sigma^n$ and $|\text{slope}(u)|\leq\frac{1}{10}$.
Similarly, because by construction, $\zeta_0=(\xi_0,\eta_0)$,  with
$|\xi_0|>\delta$ is $\kappa$-expanding up to time $n$ and therefore
we can apply Lemma $6.4$ of \cite{MoraViana}, that in our setting becomes:
\begin{lem}\label{lemma6.4}
Let $\zeta'_0$ be such that $|\zeta_{\nu}-\zeta'_{\nu}|\leq\sigma^{\nu}$ for every $1\leq\nu\leq n$ with $\sqrt{b}\leq\sigma\leq\left(\kappa/{10 K^2}\right)^4$. Then 
\begin{itemize}
\item[(a)]$\frac{1}{2}\leq\frac{\left\|M^{\nu}(\zeta_0)u\right\|}{\left\|M^{\nu}(\zeta'_0)v\right\|}\leq 2$,
\item[(b)]$\left|\text{\rm angle}\left(M^{\nu}(\zeta_0)u, M^{\nu}(\zeta'_0)v\right)\right|\leq \left(\frac{K^2\sqrt{\sigma}}{\kappa}\right)^{\nu+1} $
\end{itemize}
for any $1\leq\nu\leq n$ and any norm $1$ vectors $u,v$ with
$|\text{\rm slope}(u)|\leq\frac{1}{10}$ and $|\text{\rm slope}(v)|\leq\frac{1}{10}$.
\end{lem}

The above result combined with results at the end of Section 6 and
Section 7C in \cite{MoraViana} gives the following lemma on the
existence of the stable vector field $e^{(n)}$ and the corresponding
stable foliation which will be instrumental for the capture argument,
Section \ref{sec:capturing}, 
and also for the construction of the sink, Section \ref{sink}.
\begin{lem}\label{stable-foliation}
Let $\zeta_0$ satisfy equation \eqref{eq:expansion}
and let $s$ be a segment of $W^{u}(\hat z)$ centered in $\zeta_0=$  of length $\sigma^{2n}$. The stable vector field $e^{(n)}$ through $s$ can be integrated from $s$ to $G_1=F(G_0)$. Let $s_1$ be the arc of end points obtained on $G_1$, then 
\begin{itemize}
\item[(a)]$\text{\rm dist}\left(F^n(s),F^n(s_1)\right)=K\kappa^n$,
\item[(b)]$\left|\text{\rm angle}\left(M^{n}(\zeta'_0)u, M^{n}(\zeta''_0)v\right)\right|\leq\left(\frac{K^2\sqrt\sigma}{\kappa}\right)^4$,
\end{itemize}
where $\zeta'_0\in s$, $\zeta''_0\in s_1$, $u=\tau(\zeta'_0)$ and $v=\tau(\zeta''_0)$.
\end{lem}
We also need Lemma 6.1. from \cite{MoraViana}. 
\begin{lem}\label{contractive-field}
If $e_{\nu}(z)$ is the most contractive direction, then for $1\leq\mu\leq\nu\leq n$
\begin{itemize}
\item[(a)]$\left|\text{\rm angle}(e_{\mu}(z),e_{\nu}(z))\right|\leq\left(\frac{3K}{\kappa}\right)\left(\frac{Kb}{\kappa^2}\right)^{\mu}$ ,
\item[(b)]
$\left\|Df^{\mu}(z) e_{\nu}(z)\right\|\leq\left(\frac{4K}{\kappa}\right)\left(\frac{K^2b}{\kappa^2}\right)^{\mu}$.
\end{itemize}
\end{lem}

We consider the integral curves of the vector field 
$$
\left(\begin{matrix}
\dot x\\\dot y
\end{matrix}\right)=e_1(z).
$$
Since $$DF(z)^{-1}=\frac{1}{\text{det}DF(z)}\left(
\begin{matrix}
D &-B\\ -C& A
\end{matrix}
\right)$$
and $A=-2ax+O(b^t)$, $C_1\sqrt b\leq |B|\leq C_2\sqrt b$, it is easy to see that 
$$
\text{slope }e_1(z)=-\frac{A}{B}\approx\frac{2ax}{\sqrt b}.
$$
\

As a conclusion we get that the integral curves of the stable vector
field $e^{(1)}$ are approximate parabolas. At the critical value
$z_1$, the expansive property  \eqref{eq:expansion} is valid and  we
obtain the following result, see Figure 1.

\begin{lem}\label{lem:stablefoliation}
Suppose that $F$ satisfies  the assumption of Lemma \ref{stable-foliation}. Then there is a quadrilateral containing 
the critical value, which is completely foliated with leaves that are
integral curves of $e_k(z)$ given that $k=\left[\frac{n}{10}\right]$.
\end{lem}

\begin{proof}
This is a small variation of Lemma $5.8$ in \cite{BC2},
which we are going to pursue in the following with more detail. The idea
is to successively define smaller and smaller quadrilaterals $Q_n$
which are foliated by integral curves of the  most contractive
vector field $e_k(z)$ of $DF^k(z)$. 

We know that for the point $\tilde z_0=z_1$
$$
\left|DF(z_1)\begin{brsm}1\\0\end{brsm}\right|\geq e^{\tilde \kappa\nu},\qquad \nu=1,\dots, n.
$$
Moreover we will only use this estimate in the range $1\leq\nu\leq k$, $k=\left[\frac{n}{10}\right]$. We will inductively define a sequence $\left\{\gamma_{i}\right\}$ of integral curves of $e_{i}(z)$ through $z=z_1$. We start by defining $\gamma_1$ as the integral curve of $e_1(z)$ through $z_1$. We now pick $\tilde z_0=z_1$. Suppose $\gamma_{i}$ is defined and stretches from $y=-1$, $y=1$. Pick a point $\zeta_0\in\gamma_{i}$. Then by Lemma 6.1 $(b)$ in \cite{MoraViana}, 
$$
d(\zeta_j,\tilde z_j)\leq \left(\frac{4K}{\kappa}\right)\left(\frac{K^2b}{\kappa^2}\right)^j.
$$
Let $\zeta'_0$ be on the horizontal segment containing $\zeta_0$ at distance $\left(\frac{4K}{\kappa}\right)\left(\frac{Kb}{\kappa^2}\right)^{i}$,
\begin{eqnarray*}
d(\zeta'_j,\tilde z_j)&\leq & \left(\frac{4K}{\kappa}\right)\left(\frac{K^2b}{\kappa^2}\right)^j+5^j\left(\frac{Kb}{\kappa^2}\right)^{i}\\
&\leq &\left(\frac{8K}{\kappa}\right)\left(\frac{K^2b}{\kappa^2}\right)^j.
\end{eqnarray*}
Define 
$$
\Omega_{i}=\left\{ z \left| \right. \text{dist}_{\text{h}}(z,\gamma_{i})\leq 16 K\left(\frac{Kb}{\kappa^2}\right)^{i}\right\}.
$$
Then the integral curves of $e_{i+1}(z)$ are defined in $\Omega_{i}$ and do not leave $\Omega_{i}$. We define $\Omega_{i+1}$ by the restrictive condition 
$$
\Omega_{i+1}=\left\{ z \left| \right. \text{dist}_{\text{h}}(z,\gamma_{i+1})\leq 16 K\left(\frac{Kb}{\kappa^2}\right)^{i+1}\right\}.
$$
We proceed in this way by induction. Finally we can vary the point
$\tilde z_0$ on a horizontal line segment $s$ through $z$, providing
that $|s|\leq c^n$ (for a suitably choosen $c$).
\end{proof}

 \begin{figure}[h]
 \centering
 \includegraphics[width=0.9\textwidth]{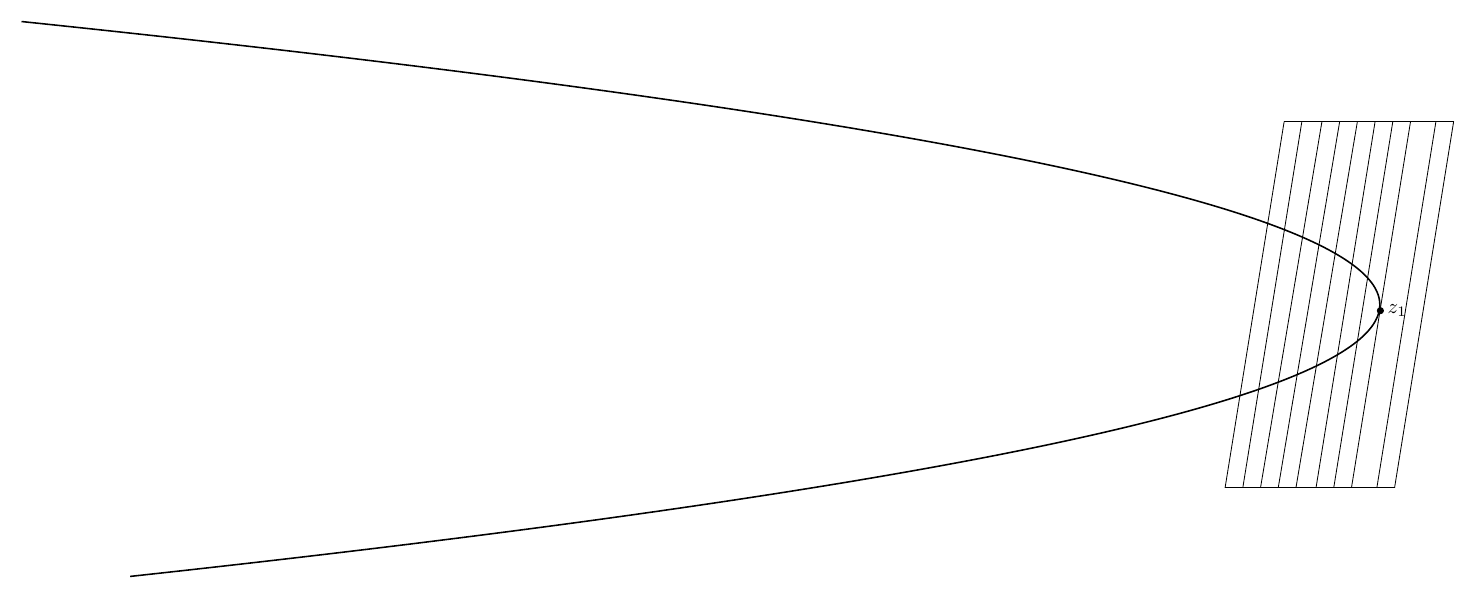}
 \caption{Stable foliation at the critical value}
 \label{Fig1New}
 \end{figure}








\section{Construction of a sink}\label{sink}
In the following we work in the H\'enon-like setting.
Let $z_0\in\Gamma_E$ be the critical point on the left leg of
$W^u(\hat z)$, see Subsection \ref{ss:twod}. One can choose $z_0$ uniquely for all
$a\in\omega_0\in\mathcal P_E$, see Section $5$ in \cite{MoraViana} or
Section $6$ in \cite{BC2}. We nox fix $E_0$ to be such that
$z_{E_0}(\omega_0)$ is in an escape situation as defined in the end of
Subsection \ref{onedimensional-case}. 
\subsection{Construction of a long escape situation}\label{sec:long_escape_situation}
The aim of this section is to prove that long escape situations
occur. In these situations we can guide the dynamics to behave in the
direction we wish, in particular, we can create attractive periodic orbits.

\begin{defin}
We say that $z_{E}(\omega)$, $\omega\subset\mathcal P_E$,  is in a long escape situation at time $E$ if $z_{E}(\omega)$ is a $C^2(b)$ curve\footnote{See Definition \ref{C2bcurve}} such that $$\pi_1 z_{E}(\omega)\supset\left[\frac{3}{8},\frac{5}{8}\right],$$ where $\pi_1$ is the projection on the first coordinate, i.e. if $\gamma(t)=(\gamma_1(t),\gamma_2(t))$ then $\pi_1\gamma(t)=\gamma_1(t)$.
\end{defin}
\begin{lem}\label{lem:long_escape_sit}
There exist $\tilde\omega_0\subset\omega_0$ and a time $E$ such that $z_{E}(\tilde\omega_0)$ is in a long escape situation.
\end{lem}
\begin{proof}
This proof is purely one-dimensional, since $b$ is small and the dynamics is outside of $\left(-\delta,\delta\right)\times\R$. We use an argument very similar to that in \cite{Th}. By \cite{MoraViana}, there is a time $n$ and an interval $\omega_0\in\mathcal P_n$ so that $\pi_1 z_{n}(\omega_0)\cap (-\delta,\delta)\neq\emptyset$ and $\left|\pi_1 z_{n}(\omega_0)\right|\geq\sqrt{\delta}$. Consequently, one of the components, $L_n'$ of $\pi_1 z_{n}(\omega_0)\setminus\left(-\delta,\delta\right)$ has length bigger than $\sqrt{\delta}/3$. Let $\omega'=\left[a_1,a_2\right]$ be defined by the relation $$\pi_1 z_n(\omega')=L_n'=\left[\pi_1 u,\pi_1 v\right],$$
where $u$ and $v$ are the end points of the curve $z_n(\omega')$.
Consider then the future iterates $z_{n+i}(\omega')$, $i=1,2,\dots,$ under the parameter dynamics. Observe that $\pi_1 z_{n+2}(\omega')$ is located at \begin{eqnarray*}
\left(\pi_1 F^2_{a_1}(u),\pi_1 F^2_{a_2}(v)\right)&=& \left(1-a_1\left(1-a_1\delta^2\right)^2+O(b^t),\pi_1 F^2_{a_2}(v)\right)\\
&=& \left(1-a_1+O\left(\delta^2\right)+O(b^t),1-a_2+\Theta\left(\delta^{\frac{4}{3}}\right)\right),
\end{eqnarray*}
where the function $\Theta(x)$ satisfies $c_1 x\leq \Theta(x)\leq c_2
x$ for some numerical constants $c_1$ and $c_2$. Observe that
$F^2_{a_1}(u)$ and $F^2_{a_2}(v)$ and consequently $\left(\pi_1
  F^2_{a_1}(u),\pi_1 F^2_{a_2}(v)\right)$ are located near the saddle
fixed point close to $(-1,0)$ where the dynamics is expanding in the $x$-direction by a factor bigger than $3$ as long as 
\begin{equation}\label{closecritpoint}
\pi_1 F^{2+i}_{a_2}(v)\leq-\frac{3}{4}
\end{equation}
Denote by $i_0$ the last $i$ for which (\ref{closecritpoint}) is verified. Then $\pi_1 F^{2+i_0}_{a_1}(u)$ is still close to $-1$; its distance to $-1$ is of order $O\left(\delta^{2-\frac{4}{3}}\right)$. After $2$ more iterates 
 $$\left(\pi_1 F^{4+i_0}_{a_1}(u), \pi_1 F^{4+i_0}_{a_2}(v)\right)\supset\left[\frac{3}{4},\frac{5}{4}\right].$$
\end{proof}

To the fixed point $(\hat{x},\hat{y})$ there is a symmetric point on
$W^u(\hat{x},\hat{y})$, $(\hat{x_1},\hat{y_1})$,  located approximately
at $(-\hat{x},\hat{y})$. The  leg of $W^s(\hat{z})$ in the negative
$y$-direction  crosses
this homoclinic point and the slope $s$ of the curve segment of
$\gamma_s$ joining the
two points  $(\hat{x},\hat{y})$ and    $(\hat{x_1},\hat{y_1})$
satisfies $s\geq C/\sqrt{b}$ on all points of $\gamma_s$, see Lemma
\ref{stablemanifold}. We choose the intersection with the preimage to
ensure that at the next iterate when the curve segment intersects the
stable manifold, the distance to the fixed point $\hat{z}$ is defined
by a high accuracy and is very close to the width of the parabola at
this $x$-coordinate. This is needed to make the time $E'$, which
will appear later, well defined, see Lemma \ref{lambda2}. 

\begin{lem}\label{homoclinic} 
There is a subinterval $\tilde\omega_0'\subset\tilde\omega_0$ such
that, for all $a\in\tilde\omega_0'$, the stable leg of $W^s(\hat{z})$
pointing downwards, denoted by $\gamma^s_a$, intersects the middle half of $z_E(\tilde\omega_0')$.
\end{lem} 
\begin{proof}
Let $\tilde a_0$ be the midpoint of $\tilde\omega_0$ and let
$p_1=\gamma^s_{\tilde a_0}\cap z_E(\tilde\omega_0)$. Let $\tilde a_0'$
be the preimage of $p_1$ in $\tilde\omega_0$. Observe that
$\gamma^s_{\tilde a_0'}$ intersects $z_E(\tilde\omega_0)$ at $p_2$. By
Lemma \ref{par-phase-dist1},
$$
|p_1-p_2|\leq K |\tilde\omega_0|\leq K
e^{-cE},
$$
where $K$ is a positive constant. We choose now a subinterval $\tilde\omega_0'\subset\tilde\omega_0$ having midpoint $\tilde a_0'$ and such that $z_E(\tilde\omega_0')$ has length $e^{-c E}$. Then $\tilde\omega_0'$ has the required property, i.e. for all $a\in \tilde\omega_0'$, $\gamma^s_a$ intersects $z_E(\tilde\omega_0')$ in its middle half.
\end{proof}
The following lemma allows us to control the dynamics so that part of
the parameter interval  returns close to a critical point with a
controlled geometry, see Figure \ref{Fig0}.  This will create an attractive periodic orbit for
all selected parameters.

\begin{lem}\label{returnlemma}
There is a subinterval $\tilde\omega_0''\subset\tilde\omega_0',$ with midpoint $\tilde a_0''$ and a time $N$ so that, $z_N(\tilde\omega_0'')$ has the following properties:
\begin{itemize}
\item[(i)] $z_{N}(\tilde\omega_0'')$ is a $\Cd(b)$ curve,
\item[(ii)] $|z_{N}(\tilde\omega_0'')|=\frac{1}{100}\frac{1}{D_N}$,
\item [(iii)] $\text{\rm dist}\left(\pi_1 z_0(\tilde a_0''),z_N(\tilde\omega_0'')\right)\leq\frac{1}{50}\frac{1}{D_N}$,
\end{itemize}
\end{lem} 
where $D_N=|w_N|$.
\begin{figure}[h]
\centering
\includegraphics[width=0.9\textwidth]{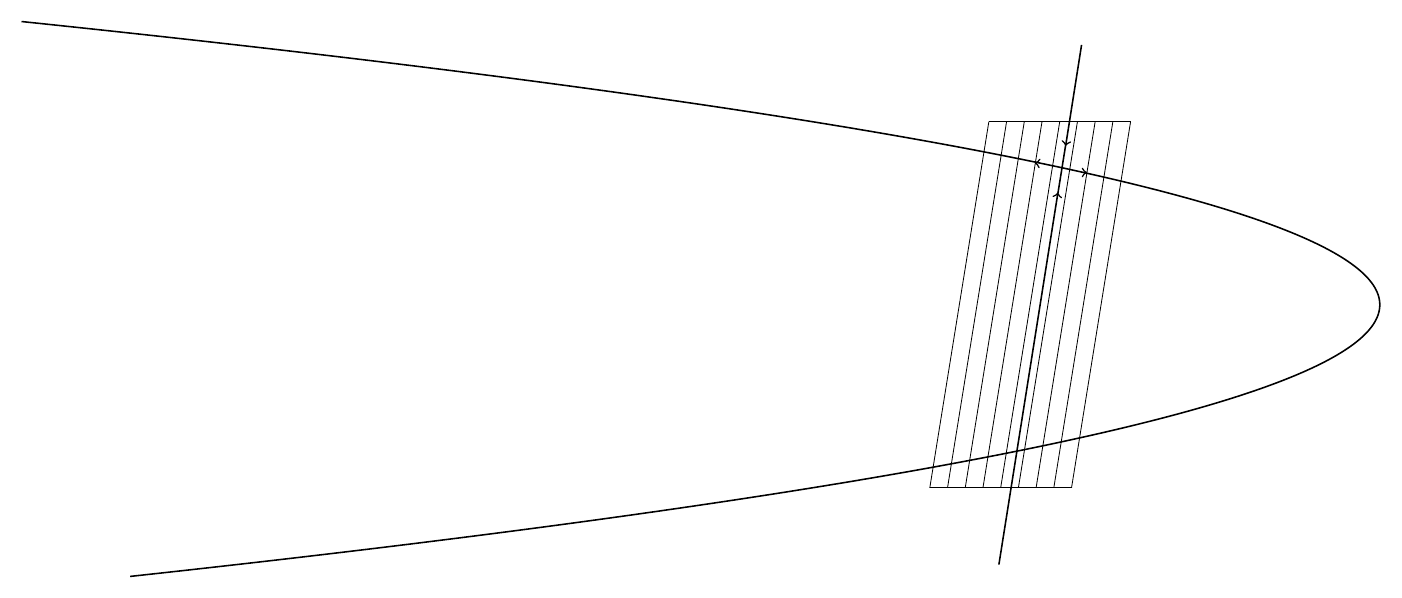}
\caption{Stable foliation at the fixpoint}
\label{Fig_stable_fixedpoint}
\end{figure}

\smallskip
  The proof of Lemma \ref{returnlemma} consists of several steps,
  formulated in a sequence of lemmas.

Consider the phase curve $\gamma=z_E(\tilde\omega_0')$ and denote by
$\tilde a_0'$ the midpoint of $\tilde\omega_0'$. We recall the $\lambda$-lemma, see e.g. \cite{PalisdeMelo}, Lemma 7.1.

\begin{lem}\label{lambda1}  Let 0 be a saddle fixed point of a $C^2$ map.  Let $V=B^u\times B^s$ be the cartesian product of an
  unstable and stable ball at the fixed point 0, let
  $q\in W^s(q)\setminus \{0\} $ and let $D^u$ be a disk transverse to
  $W^s$ intersecting  $W^s$ in $q$. Let $D^u_n$ be the connected
  component of $F^n(D^u)\cap V$ to which $F^n(q)$ belongs. Given
  $\varepsilon>0$ there exists $n_0\in{\mathbb N}$ such that if
  $n>n_0$, then $D^u_n$ is $\varepsilon>0$ $C^1$ close to $B^u$.
\end{lem}

  In our present setting we can obtain a quantative version of the
  $\lambda$-lemma adapted to our situation. In the following we refer
  to  Figure \ref{Fig_stable_fixedpoint}.
  
  \begin{lem}\label{lambda2}
Suppose a $C^2(b)$-curve $\gamma$  of size $e^{-\kappa E}$  crosses the leg of
$W^s(\hat{z})$ in the negative $y$-direction. Then after $E'$ iterates
where $E'\sim E$, $F^{E'}_a(\gamma)$ will be a $C^2(b)$ curve stretching along
$W^u(\hat{z})$ and across the ordinate axis $x=0$ to
$x=-\frac{1}{4}$. Close to $x=0$ the vertical distance between
$W^u(\hat{z})$ and  $F^{E'}_a(\gamma)$ can be estimated as

\begin{equation}\label{eq:dist}
\leq \text{\rm const. } \left(\lambda_s\right)^{\frac{1}{10}E'}.
\end{equation}
and the angles between points with the same $x$-coordinate satisfies

\begin{equation}\label{eq:angle}
\leq \text{\rm const. } \left(\lambda_s\right)^{\frac{1}{40}E'}.
\end{equation}
 \end{lem}

 \begin{proof} 
We apply the construction of the stable foliation in lemmas
\ref{stablefoliation} and \ref{lem:stablefoliation}. For each point of $\zeta_0\in\gamma$ we connect it to a
corresponding point $\zeta_0'$ on $W^u(\hat{z})$. It is then possible
to apply Lemma \ref{lemma6.4} with $\tilde{z}_0=\zeta_0$,
$\tilde{z}_0'=\zeta_0'$ and $\kappa=(1+\varepsilon)\lambda_s$, for a
suitable $\varepsilon>0$. 
We conclude that the estimates of \eqref{eq:dist} and \eqref{eq:angle} hold.
\end{proof}

\begin{rem} Note that $\lambda_u\cdot \lambda_s=\det DF_a(\hat{z})$
and that the factor $\frac{1}{10}$ comes from the comparison between
$\kappa$ and $\log |\lambda_u|$, where  
$\log 2-\varepsilon \leq \log |\lambda_u| \leq \log2$, and where $\varepsilon$ depends on $2-a$.
\end{rem}

\medskip
{\em Proof of Lemma \ref{returnlemma}}. For the following we refer to
Figure \ref{Fig0} .

\begin{figure}[h]
\centering
\includegraphics[width=0.9\textwidth]{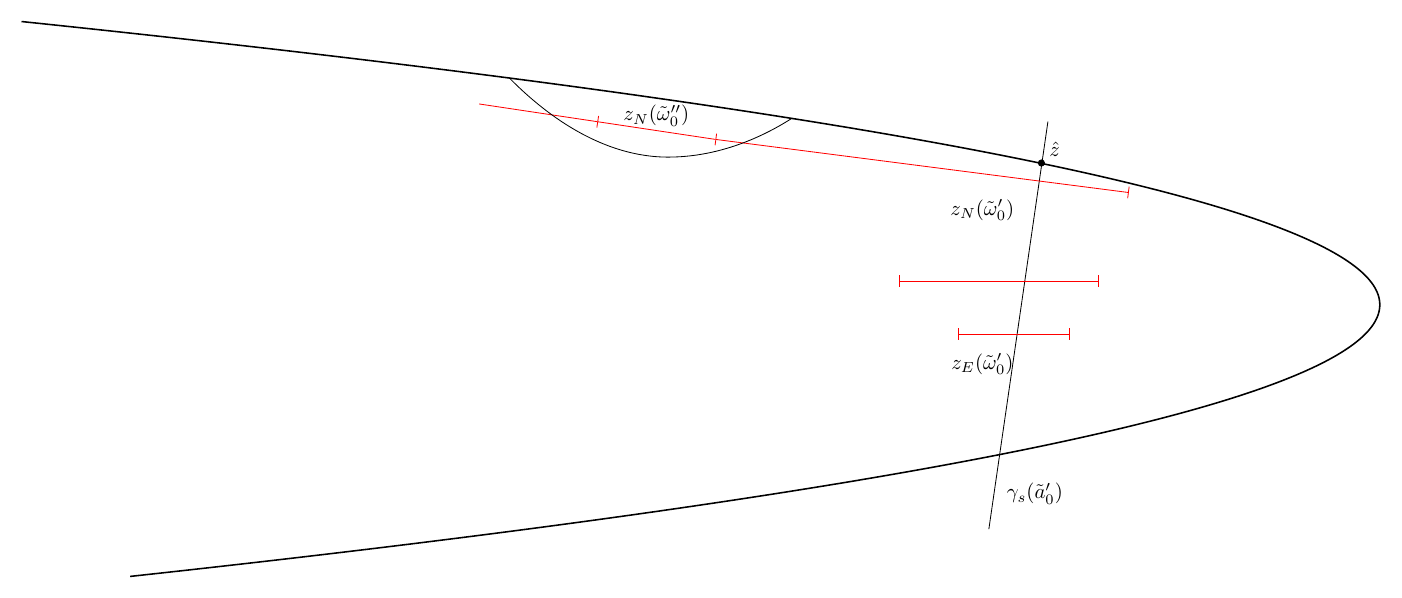}
\caption{The capturing argument}
\label{Fig0}
\end{figure}

\begin{itemize}

\item[(i)] We apply Lemma \ref{lambda1} to a fixed parameter $\tilde{a}_0'\in\tilde{\omega}_0'$
from Lemma \ref{stable-foliation}, (b),  
to $\gamma$ with fixed
parameter $\tilde a_0'$. At a certain time $E'\sim E$, $F_{\tilde a_0'}^{E'}(\tilde{\omega}_0')$ stretches along $W^u(\tilde a_0')$ covering its
$x$-projection $\left[-\frac{1}{4},\frac{1}{4}\right]$.

\item[(ii)] By the
comparability of $x$ and $a$ derivatives, see Corollary \ref{cor:pardist},  during the
time from $E$ to $E+E'$ and the fact that $|\tilde\omega_0'|\sim
e^{-2cE}$, one can check that $z_{E+E'}(\tilde\omega_0')$ covers the
$x$-projection  $\left[-\frac{1}{8},\frac{1}{8}\right]$. Now restrict
$\tilde\omega_0'$ to a subinterval $\tilde\omega_0''$ with midpoint
$\tilde a_0''$ so that for $N=E+E'$,
$|z_{N}(\tilde\omega_0'')|=\frac{1}{100}{D_N}^{-1}$.

\item[(iii)]
Note that, as in
\cite{MoraViana}, Section 7, $z_{N}(\tilde\omega_0'')$ is a $\Cd(b)$
curve and $\text{\rm dist}\left(\pi_1 z_0(\tilde
  a_0''),z_N(\tilde\omega_0'')\right)\leq\frac{1}{50}D_N^{-1}$
and we also obtain by Lemma \ref{stable-foliation}, (b),
\eqref{eq:angle} that the angle $\theta$ between the points of
$z_N(\tilde\omega_0'')$
with the same $x$-coordinate on the first leg of $W^u(\hat{z})$
satisies  
\begin{equation}\label{eq:angle1}
\theta\leq \text{\rm const. } \left(\lambda_s\right)^{\frac{1}{40}E'}.
\end{equation}
Here we again have to use the comparasion of parameter and phase
derivatives, Lemma \ref{par-phase-dist1} and the distorsion of the the
$a$-derivative within a partition interval, see Corollary \ref{cor:pardist}.

\end{itemize}


\subsection{Construction of an invariant contractive region}
In this section we prove the existence of an invariant contractive region around the critical point. 
We pick an arbitrary $a\in\tilde\omega_0''$, with $\tilde\omega_0''$
as in Lemma \ref{returnlemma}. We refer to Figure \ref{Fig2New}.
\begin{figure}[h]
\centering
\includegraphics[width=0.9\textwidth]{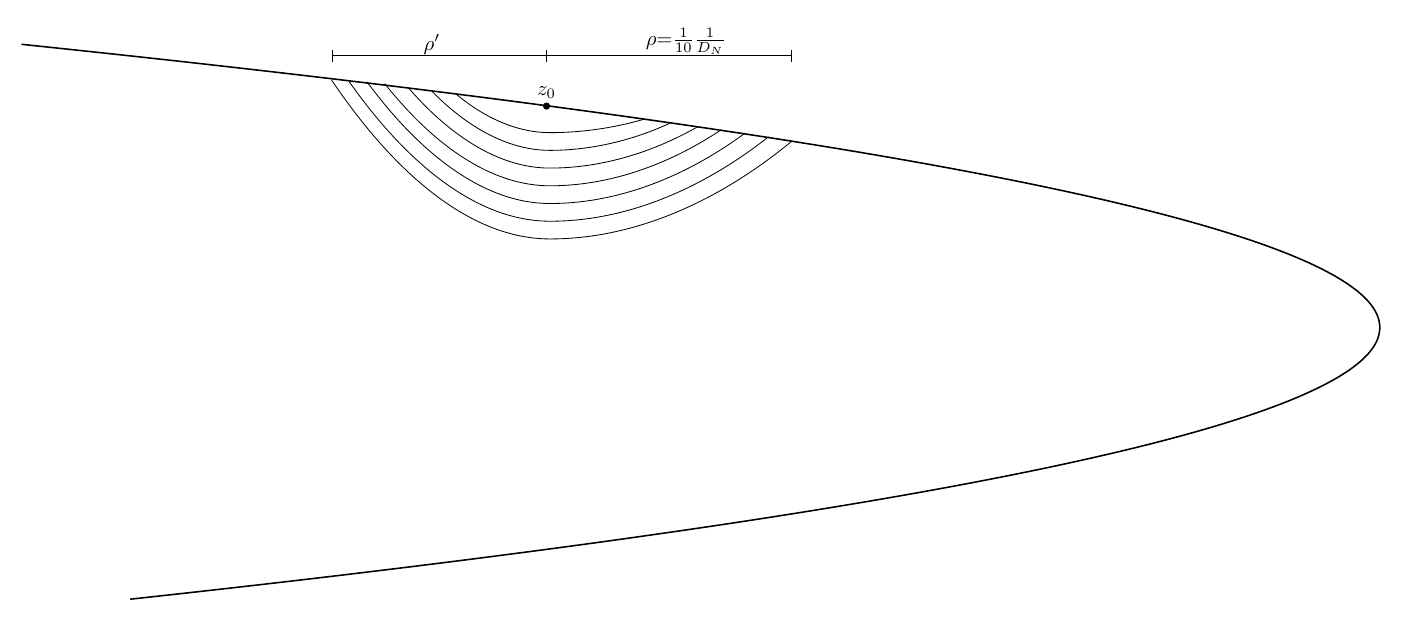}
\caption{Stable foliation at the critical point}
\label{Fig2New}
\end{figure}

Associated to $a$ there is a critical
  point $z_0(a)$ located on the first left leg of $W^u(\hat z)$, see
  Subsection \ref{ss:twod}. We
  fix now a curve $\gamma :(-\rho',\rho)\to\mathbb{R}^2$ on this left
  leg so that $\gamma(0)=z_0$, where $\rho=\frac{1}{10}{D_N}^{-1}$.
  and $\rho'$ will be choosen as follows.

  Close to the critical value $z_1$ there is, by Lemma 
  \ref{lem:stablefoliation}, a quadrilateral foliated by leaves of the
  stable vector field $e_{[N/10]}$. The leave $\gamma_3'$ of $e_{[N/10]}$ through
  $F(\gamma(\rho))$ hits $W^u(\hat{z})$ in another point $\zeta'$ and
  $\rho'$ is defined so that $F(-\rho')=\zeta'$. The pullback of the
  stable leave $\gamma_3'$ by $F$ is denoted by $\gamma_3$.

We define
  $\mathcal{D'_N}$ as the domain bounded by
  $f\left(\gamma_{|(-\rho',\rho)}\right)$ and the stable leave
  $\gamma_3'$. Let $\mathcal D_N$ be the pullback under $F$, namely $\mathcal
  D_N=F^{-1}\left(\mathcal D'_N\right)$. We will prove that $\mathcal
  D_N'$ and hence also   $\mathcal D_N$ are invariant under $F_a^N$
  for all $a$ in  $\tilde\omega_0''$.

 Consider the tangent vector $\tau_1(s)$ of $\gamma_1(s)=F_a(\gamma(s))$ and write it, following Lemma 9.6 in \cite{MoraViana} as 
$$
\tau_1(s)=\alpha(s)e_{E-1}(s)+\beta(s)w_1,
$$
with $\frac{3}{2}a|s|\leq \left|\beta(s)\right|\leq\frac{5}{2}a|s|$ and $w_1=\begin{brsm}1\\0\end{brsm}$. Observe that, at time $E$,
$$
\left\|DF^{E-1}_ae_{E-1}\right\|=O\left(b^{E-1}\right).
$$
Denote by $\gamma_E^1$ and
$\gamma_E^2$ the two sub-curves of $\gamma$ defined by restricting the
arclength to $(-\rho',0)$ and
$(0,\rho)$ respectively. For the image of these curves the tangent vector decomposes as
$$
\tau_E(s)=\alpha(s)DF^{E-1}e_{E-1}(s)+\beta(s)w_{E-1}.
$$
Since, by the induction, $\|w_E\|\geq e^{\kappa E}$, we conclude that

$$
\left|\alpha(s)DF^{E-1}\left(e_{E-1}(s)\right)\right|\leq O(b^{E-1})\leq\frac{1}{2}|s|\|w_E\|
$$
 and since $\text{slope}(w_E)=O(b^t)$, it follows that
 $\gamma^1_E\setminus\tilde{\gamma}^1_E  $ and
 $\gamma^2_E\setminus\tilde{\gamma}^2_E $ are $C^2(b)$ curves. The
 curves $\tilde{\gamma}^1_E$ and $\tilde{\gamma}^2_E$
 correspond to the subsegments close to $z_E$, which are still in fold
 periods of the initial binding to $z_0$, and those segments are of
 size $(Cb)^E$. The curve $\gamma^3_E=F^E(\gamma_3)$ has, by Lemma
 \ref{contractive-field} (b), length $|\gamma^3_E|\leq (Cb)^E$.

There is, by Lemma \ref{contractive-field}, a stable vector field
$e_{E'}$ defined in a vertical region containing the curves
$\gamma^1_E $, $\gamma^2_E$ and $\gamma^3_E$.
By
\cite{BC2}
the curves $F^{E'}(\gamma^1_E)$, $F^{E'}(\gamma^2_E)$ and
$F^{E'}(\gamma^3)$ are located below $\gamma$ and at distance
$O(b^{E'})$. By the angle estimate
\eqref{eq:angle1} it follows  that except for the points still in fold period to $z_0$ at
time $N=E+E'$, the slopes of points of the curves
$\gamma'=F^{E'}(\gamma^1_E)$ and
$\tilde{\gamma}'=F^{E'}(\gamma^2_E)$ with the same
$x$-coordinates is $\leq (Cb)^{E'/40}$.

The curve $F^{E'}(\gamma^3)$ has diameter $\leq 2\cdot 5^{E'}\cdot (Cb)^E$, and it is located close to $z_N$. At this point we choose $\rho'$ so that $F(\gamma(\rho))$ and $F(\gamma(-\rho'))$ are on the same stable leave of $e_E$ close to $\hat z$. The curve segment $F^N(\gamma^1)$ has length 
\begin{eqnarray*}
\text{length}(F^N(\gamma^1))&\leq &\int_{0}^{\rho}|\beta(s)|\|w_N(s)\| ds+\int_{0}^{\rho}O(b^N)d\rho\\
&\leq &\int_{0}^{\rho}4s D_N ds+O(\rho b^N)=2\rho ^2D_N+O(\rho b^N)\\
&\leq &3\left(\frac{1}{10}D_N^{-1}\right)^2\cdot D_N=\frac{3}{100}\frac{1}{D_N}.
\end{eqnarray*}
The length of $F^N(\gamma^2)$ is estimated similarly. Finally
$$
\text{diam}(F^N(\gamma^3))\leq 5^{E'}(Cb)^E\leq\frac{2}{100}\frac{1}{D_N}.
$$
 It follows that $F^{N-1}\left(\mathcal D'_N\right)$ has diameter $\leq\frac{5}{100}D_N^{-1}$ and it is at distance $O(b^{N-1})$ to $\gamma$. Since $\|DF\|_{\Cuno}\leq 5$, then 
$$
F^{N}\left(\mathcal D'_N\right)\subset \mathcal D'_N.
$$
The discussion above can be summarized in the following lemma (see
Figure \ref{Fig6}).
\begin{lem}
For all     $a\in\tilde\omega_0''$ , there exists a domain ${\mathcal
  D}_N(a)$ around the critical point $z_0(a)$, so
that $$F_{a,b}^N\left(\mathcal D_N(a)\right)\subset\mathcal D_N(a).$$
A corresponding statement holds for the region $\mathcal D_N'(a)$
close to the critical value $F_a(z_0)$
\end{lem}
\begin{figure}[h]
\centering
\includegraphics[width=0.9\textwidth]{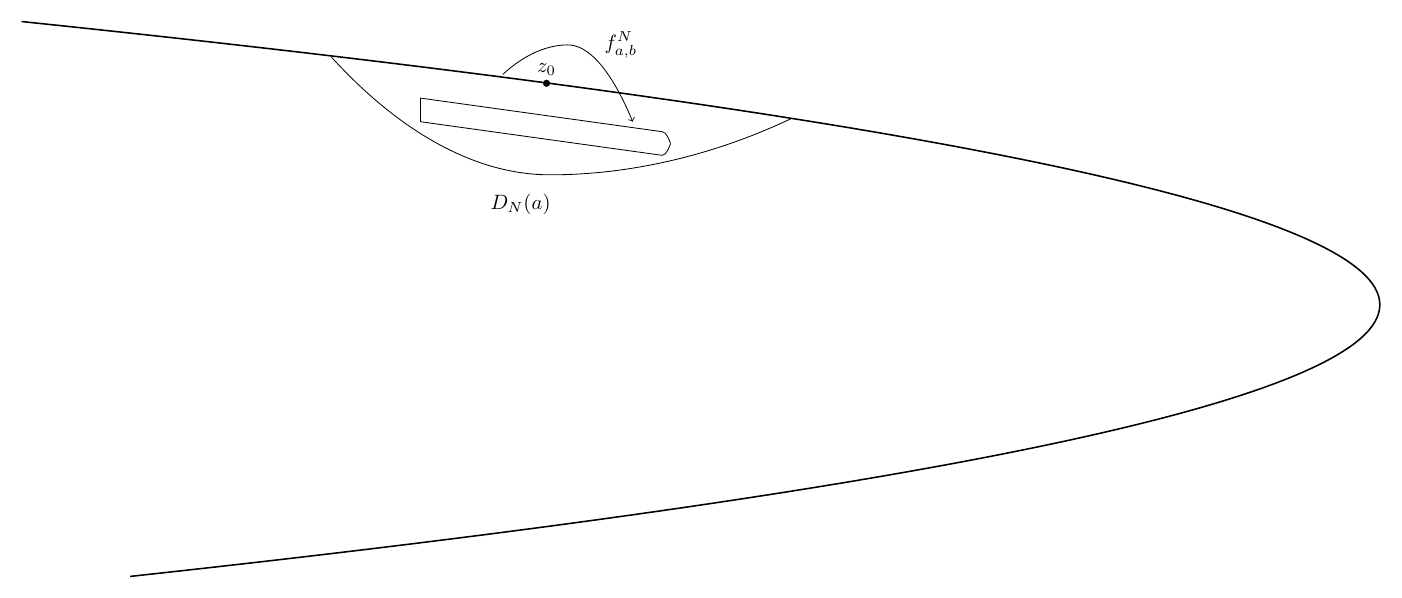}
\caption{The invariant region at the critical point}
\label{Fig6}
\end{figure}

\begin{lem}\label{contraction}
There exists an integer $k$ such that, for all $a\in\tilde\omega_0''$, $F_{a,b}^{Nk}$ contracts.
\end{lem}

\begin{proof}
Take an arbitrary point $z\in\mathcal D'_{N}(a)$ and as in Subsection
\ref{ss:splitting}, consider the unit vector $$v=\alpha_0 e_n(z)+\beta_0w_1,$$ where $w_0=\begin{brsm}1\\0\end{brsm}$
and $e_n(z)$ is the contracting direction of
order $n=\left[\frac{N}{10}\right]$ at $z$. 
Consider the decomposition of $DF^N(z)v$ as $$DF^N(z)v=\alpha_0 DF^N(z)e_n(z)+\beta_0 w_{N+1}.$$
Observe that, at the first return time $N$, $e_n(z)$ is mapped to $DF^{N}(z)e_n(z)$ with 

\begin{equation}\label{contres}
\left\|DF^N(z)e_n(z)\right\|\leq 5^{N-n}b^n.
\end{equation}

Let us decompose $\alpha_0 DF^N(z) e_n(z)$ as 

$$\alpha_0 DF^N(z) e_n(z)=\alpha_1^s e_n\left(F^N(z)\right)+\beta_1^s
w_1,$$ where, by (\ref{contres}), $|\alpha_1^s|,|\beta_1^s|\leq
5^{N-n}b^n |\alpha_0|$.

Observe now that $\left\|DF^Nw_1\right\|= D_N$. As a  consequence
$$DF^N(z)\beta_0 w_0=\alpha_1^u e_n\left(F^N(z)\right)+\beta_1^u
w_0,$$
where $|\alpha_1^u|\leq D_N|\beta_0|$ and
$|\beta_1^u|\leq\frac{5}{10}\frac{1}{D_N}D_N |\beta_0|$.
Using the  notation $\alpha_\nu=(\alpha_\nu^u,\alpha_\nu^s)$,
$\beta_\nu=(\beta_\nu^u,\beta_\nu^s)$, it follows that 

 $$\left\{
 \begin{matrix}
 |\alpha_1|&\leq&|\alpha_1^s|+|\alpha_1^u|&\leq&5^{N-n}b^n|\alpha_0|+D_N|\beta_0|,\\
 |\beta_1|&\leq&|\beta_1^s|+|\beta_1^u|&\leq& 5^{N-n}b^n|\alpha_0|+\frac{5}{10}|\beta_0|.
 \end{matrix}
 \right.$$

 Let $A$ be the matrix 
Observe that $A$ has spectral radius at most $\frac{1}{2}$.
 Finally we choose $k>0$ such that $\left(\frac{1}{2}\right)^kD_N^2<1$. Then $A^k$ is a contraction and therefore also $DF^{Nk}$ is a contraction. 
\end{proof}

\section{Capturing of a new critical point}\label{sec:capturing}
The next step in the construction is to create a new attractor for the
same parameter values of maps with a sink, see Section
\ref{sink}. This attractor can be another sink or a strange
attractor. In order to do so, we need to select another critical point
and follow its evolution for the same parameter values as those of the  
first sink constructed in the previous section.


It is
important that we can use the binding critical points for the intitial
critical point. By chosing its distance appropropriately
$z_\nu(\omega)$  will
follow the intitial critical point and the new critical point  will
still be bound to the first at its first return time $N$. At this time
there will be a secondary bound period after which the secondary
critical point again is bound. After the third bound period we will
essentially be in a situation corresponding to the intial inductive
situation in \cite{BC2}, \cite{MoraViana}. Using the machinery of
\cite{BC2}, we will prove that the new critical point also will reach
an escape situation. At this point we will be able to choose
parameters which go through an unfolding of a homoclinic
tangency. Following  \cite{PalisTakens} and \cite{MoraViana}, this
will allow to create a new Henon-like family and to consequently  set
up the inductive procedure. More precisely, to this new Henon-like
family, one could apply Section 3 to create a new sink or \cite{MoraViana} to
create a strange attractor. 
\begin{figure}[h]
\centering
\includegraphics[width=0.9\textwidth]{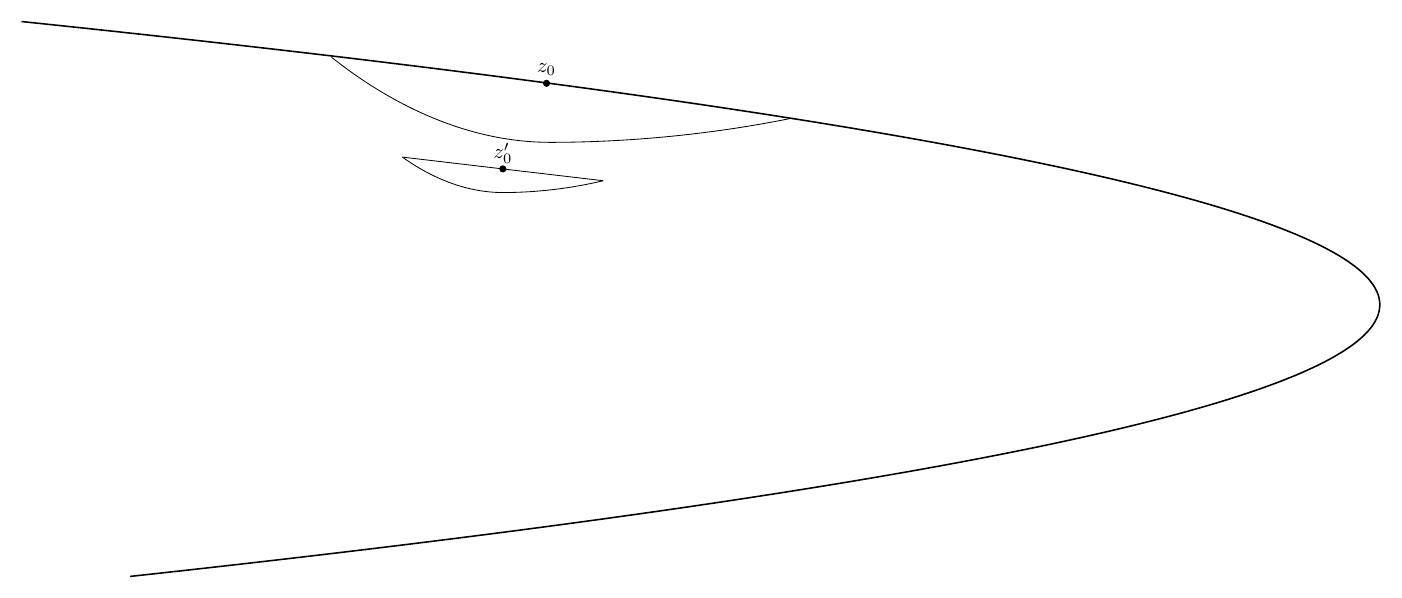}
\caption{Capturing of the second critical point}
\label{Fig3New}
\end{figure}

Our aim is first to capture a new critical point $z'_0$ at a specific distance
to $z_0$. We will show that the critical point $z_0$ and the
segment $W^u(\hat z)$ are accumulated by leaves of $W^u(\hat z)$ which
contain other critical points.
Fix $a\in\omega=\tilde{\omega}_0''$ and let $z_0=z_0(\omega)$ be a
critical point. We select a segment $L$ of the unstable manifold of
length $2\sigma^{n_1}$ aroung $\hat{z}'$, see Lemma \ref{stablemanifold}, where $n_1$ is a prescribed 
integer. By Lemma \ref{stable-foliation} and Lemma
\ref{contractive-field} it follows that the image $F^{n_2}(L)$ has
length $\approx 2\sigma^{n_1}\cdot(2a)^{n_2}$. By adjusting $n_1$ and
$n_2$, we obtain a sequence of long leaves $\gamma_j$ which accumulate
on the first leg of $W^u(\hat{z})$ restricted to $-\frac12\leq x \leq \frac12$.

This is formulated in the next lemma, where ${\rm dist}_{\rm
	v}(\hat{z}_0,z_0)$ denotes the vertical distance between the leaves of the unstable manifold containing the critical points $\hat{z}_0$ and $z_0$.

\begin{lem}\label{lem:newcapture} There are constants $C_1$, $C_2$ such that
	for  all $j\geq 16$ there is a
	critical point $\hat{z_0}$ and a corresponding segment $\hat{\gamma}^u$ containing $\hat{z_0}$

	\begin{equation}\label{eq:spacingofwuleaves}
	C_1\left(\frac{\hat{d}}{2a}\right)^{j+1}\leq \text{\rm dist}_{\rm
		v}(\hat{z}_0,z_0)\leq C_2\left(\frac{\hat{d}}{2a}\right)^{j}
	\end{equation}
where $\hat{d}=\det DF(\hat z)$.	
\end{lem}
\begin{proof} The exact estimates of \eqref{eq:spacingofwuleaves} is obtained since most of the time is spent in the linearization domain of the saddle point $\hat{z}$ where the eigenvalues are $\sim 2a$ and $\sim \hat{d}/2a$  
	
\end{proof}
\subsection{The new critical point}

Observe that, for each $n$, $\gamma_n$ and $\mathcal F^s_p$ intersects
in a unique point, $z'_0$ and that $p$ depends on $n$. Pick $n$ so
that the vertical
distance $$d_v(\gamma_u,\gamma_n)=d_n=\frac{1}{D^{\eta}_N}$$ for a
suitable $\eta$ satisfying $1<\eta<2$ to be chosen later. 
Moreover, by Lemma \ref{contractive-field}, $(b)$, there exists a constant $K$ close to $1$ so that
$$
\frac{1}{K}\leq\frac{\max_{\pi_1\gamma_n}\left|h_u(x)-h_n(x)\right|}{\min_{\pi_1\gamma_n}\left|h_u(x)-h_n(x)\right|}\leq K
$$
where $h_u$ and $h_j$ are the graphs of $\gamma_u$ and $\gamma_n$ and $\pi_1\gamma_n$ is the projection of $h_n$ on the $x$-axe.
\begin{lem}
Suppose that the horizontal distance satisfies 
$$
d_h(\gamma_u,\gamma_n)=d_n,
$$
then 
$$
d_h(z_0,z_0^{(n)})\leq\sqrt{d_n}
$$
\end{lem} 
\begin{proof}
This is a reformulation of Lemma $5$, Section $2.3.1$ of \cite{BenedicksYoung1} and the same proof applies also in our setting.
\end{proof}
\begin{lem}\label{lem:secondbd}
At time $N$, $z'_N(\omega)=F^{N}(z'_0)$ is located in horizontal position to $z_0$. Moreover there exists a constant $K$ close to $1$ so that
$$
\frac{1}{K}d_h(z_0,z'_N)\leq d_h(z_N,z'_N)\leq K d_h(z_0,z'_N).
$$
Furthermore
$$
\frac{1}{K_1}D_N^{1-\eta}\leq d_h(z_0,z'_N)\leq K_1 D_N^{1-\eta}
$$
for some constant $K_1$ close to $1$.
\end{lem}
\begin{proof}
Let $\Gamma_0$ be a curve joining $z_0$ and $z'_0$ and let $\Gamma_1$
be its image joining $z_1$ and $z'_1$ close to the critical value. On
$\Gamma_0$, using Subsection \ref{ss:splitting}, we decompose the tangent vector as
$$
\tau(z)=\alpha(y)e_N(z)+\beta(y)\begin{brsm}1\\0\end{brsm}
$$
with $z=(x,y)\in\Gamma_0$. Consider now the vertical segment from $z_0$ to $\gamma_n$ and let $y_n,y'_n$ be the $y$-coordinates of its end points. Then 
$$
\frac{1}{K}d_n\leq\int^{y_n}_{y'_n}\beta(y)dy\leq K d_n
$$ 
with $K$ a constant close to $1$. Use the notation $w_j=DF^j(z_0)\begin{brsm}1\\0\end{brsm}$ and apply the distortion estimates during the bound period for $w_j$, see Lemma $10.2$ in \cite{MoraViana}, which gives
$$
\frac{1}{K}D_N\leq\left\|w_N\right\|\leq K D_N.
$$
Furthermore 
$$
\frac{1}{K}\frac{1}{D^{\eta}_N}\leq d_n\leq K \frac{1}{D^{\eta}_N}.
$$
This proves the last inequality of the lemma. 
\end{proof}
\smallskip
Observe now that, by Corollary $5.7$ in \cite{BC2}, $w_N$ and the
tangent vector $\tau_N$ are aligned with $\gamma_u$ forming an angle
smaller than $d^4_n$. Note that Lemma $5.5$ and Corollary $5.7$ in
\cite{BC2} do not depend on the special form of the map and applies
also in our context. As final remark, one can notice that the
distortion during the bound period are stated in the case of phase
space dynamics. Moreover they are valid also in the parameter
dependent setting because of the uniform comparison between the $x$
and $a$-derivatives, see Corollary \ref{cor:pardist}.
\paragraph{The second bound period from time $N$ to time $2N$.} Note
that, for $\eta$ close to $2$, $z'_{2N}(\omega)$ will still be bound
to $z_N$ and that $z'_N(\omega)$ is located in horizontal position
with respect to $z_0$. We repeat the same procedure as in Lemma \ref{lem:secondbd}. Join $z_0$ and $z'_N(\omega)$ by a curve $\Gamma'_0$ and decompose the tangent vector of $\Gamma'_1=F(\Gamma'_0)$ as 
$$
\tau(s)=A(s)e_N(s)+B(s)\begin{brsm}1\\0\end{brsm},
$$
where $B(s)$ satisfies $\frac{3a}{2}s\leq B(s)\leq\frac{5a}{2}s $, see Lemma $9.6$ in \cite{MoraViana} and Assertion $4(c)$ in \cite{BC2}. Again by the bound distortion lemma in \cite{MoraViana} (Lemma $10.2$), $d(z_N,z'_{2N}(\omega))$ and $d(z_0,z'_{2N}(\omega))$ can be estimated from below and above using
$$
\frac{1}{K}s^2D_N\leq\left|\left(\int_0^sB(t)dt\right)w_N\right|\leq Ks^2 D_N
$$
where $s=d(z_0,z'_{2N}(\omega))$. A similar statement for points in horizontal position appear in \cite{BC2}, Assertion $4$, $(b)$ and $(c)$ and in \cite{MoraViana}, Corollary $10.7$. We conclude that
\begin{itemize}
\item[(a)] $d(z_0,z'_{2N}(\omega))$ is comparable with a fixed constant to $\left(D^{1-\eta}_N\right)^2D_N=D_N^{3-2\eta}$,
\item[(b)] $|z'_{2N}(\omega)|$ is comparable to $|z'_{N}(\omega)|D^{1-\eta}_N D_N$, which is comparable to $D^{1-\eta}_N$.
\end{itemize}

\smallskip
Let us now study the period when  $z'_{2N+\nu}(\omega)$, $\nu\geq 0$,
is bound  to $z_{0}(\omega)$. 

We define the preliminary binding period $p_1$ as the maximal integer so that, for all $\nu\leq p_1$,
$$
\left|z'_{2N+\nu}(\omega)-z_{\nu}\right|\leq e^{-\beta\nu}.
$$
In principle $p_1$ could be infinite, but this is not the case.
\begin{lem}\label{lem:binding}
The preliminary binding period $p_1<\infty$.
\end{lem}
\begin{proof}
The proof of this fact will follow after the proof of Lemma
\ref{lem:quadratic}.   

\end{proof}

\begin{lem}\label{lem:quadratic}
Let $\rho=\left|z'_{2N+\nu}(\omega)-z_{\nu}\right|$. If $\nu\geq\nu_0$ is outside of all folding periods, then
\begin{equation}\label{prelboundper}
\frac{3a}{2}\rho^2\left\|w_{\nu}\right\|\leq \left|z'_{2N+\nu}(\omega)-z_{\nu}\right|\leq \frac{5a}{2}\rho^2\left\|w_{\nu}\right\|,
\end{equation}
where $w_{\nu}=DF^{\nu}(z_0)w_0$.
\end{lem}
\begin{proof}
We introduce an horizontal curve $\Gamma_0$ joining $z_0$ and $z'_{2N}$ with tangent vector $\tau(s)$. The lengh of $\Gamma_{\nu}=F^{\nu}(\Gamma_0)$ is equal to 
$$
\int_0^{\rho}\left\|DF^{\nu}(\Gamma_0(s))\tau_0(s)\right\|ds.
$$
We decompose 
$$
\tau_1(s)=A(s)e_{\nu-1}+B(s)\begin{brsm}1\\0\end{brsm},
$$
and then 
$$
\tau_{\nu}(s)=A(s)DF^{\nu-1}(\Gamma_0(s))e_{\nu-1}+B(s)DF^{\nu-1}(\Gamma_0(s))\begin{brsm}1\\0\end{brsm},
$$
where, by Lemma \ref{prelboundper}
\begin{equation}\label{Bsbounds}
\frac{3a}{2}s\leq \left|B(s)\right|\leq \frac{5a}{2}s,
\end{equation}
see Section $8$ in \cite{MoraViana}. We apply the splitting algoritm from Section $8$, $(i)-(v)$ in \cite{MoraViana} to $DF^{\nu-1}(\Gamma_0(s))$.
If $v$ is outside of it follows from \eqref{Bsbounds} and integrating that 
$$
\frac{3}{4}a\rho^2\left\|w_v\right\|\leq\int_0^{\rho}\left\|\tau_v(s)\right\|ds
\leq \frac{5}{4}a\rho^2\left\|w_v\right\|.
$$
We conclude that Lemma \ref{lem:quadratic} holds.

\end{proof}

{\em Proof of Lemma \ref{lem:binding}}.
By the basic assumption which is part of the induction, see Assertion 4
(ii) in Subsection \ref{ss:twod},

$$
d(z_v(a),\mathcal C)\geq e^{-\alpha v},
$$
and $\rho=d(z_v(a),\mathcal C)$. Since by the induction $||w_\nu||\geq
e^{\kappa \nu}$, $\nu=1,2,\dots,n$, it follows that $p_1<\infty$.
\demo

Suppose now at the time $p_1$ 
$$
\left\|z'_{2N+p_1+1}(\omega)-z_{2N+p_1+1}(\omega)\right\|\geq e^{-\beta(p_1+1)}.
$$

We follow an argument from \cite{BC2}, Subsection 6.2. 
It follows from the basic assumption, see Assertion 4
(ii) in Subsection \ref{ss:twod}, that 
$$
d(z_v(a),\mathcal C)\geq e^{-\alpha v}
$$
that the deepest and longest bound period for $z_j$ satisfies $\tilde p_1\leq 4\alpha p_1$. The next level bound period satisfies $\tilde p_2\leq 4\alpha\tilde p_1$. As consequence the lenght of the combined bound period of $z_{p_1}$ will be less than 
$$
\sum_\nu {\tilde p_{v}}\leq 4\alpha p_1+(4\alpha)^2p_1+\dots=\frac{4\alpha}{1-4\alpha}p_1.
$$
This means that at the time $p$,
$$
3\rho^2\left\|w_p\right\|\geq e^{-\beta p_1}\frac{1}{4^{4\alpha p_1(1-4\alpha)}}.
$$
But $p_1\leq p\leq \left(1+\frac{4\alpha}{1-4\alpha}\right)p_1$. If we chose $\beta=10\alpha$ as in \cite{BC2} we obtain
\begin{equation}\label{firstp}
3\rho^2\left\|w_p\right\|\geq e^{-\frac{3}{4}\beta p_1} 
\end{equation}
 and also
 \begin{equation}\label{firstp2}
3\rho^2\left\|w_p\right\|\geq \rho^2 e^{-\beta p}.
\end{equation}
We can choose $\beta_1$ satisfying
$$
\frac{3}{4}\beta\leq\beta_1\leq\beta
$$
so that we have the estimate
$$
\rho^2\left\|w_p\right\|\geq C^{-1}e^{-\beta_1 p}.
$$
Let us also denote $D_p=\left\|w_p\right\|$. This means that with $p$ as in \ref{firstp2}
$$
C^{-1}e^{-\beta_1 p}\leq D_p\left(D_N^{1-\eta}\right)^2\leq e^{-\beta_1 p}.
$$
On the other hand 
$$
e^{(c_1-\alpha) p}\leq D_p\leq e^{c_1 p}
$$
so we obtain that 
$$
C^{-1}D_p^{-\beta_2}\leq D_p\left(D_N^{1-\eta}\right)^2\leq CD_p^{-\beta_2},
$$
where $\frac{\beta_1}{c_1}\leq\beta_2\leq \frac{\beta_1}{c_1-\alpha}$. Hence 
$$
C^{-1}D_N^{\frac{2(\eta-1)}{1+\beta_2}}\leq D_p\leq CD_N^{\frac{2(\eta-1)}{1+\beta_2}}.
$$
Note that the estimate 
$$C^{-1}D_p^{-\beta_2}\leq\rho^2D_p\leq CD_p^{-\beta_2}$$ 
implies that 
$$C^{-1/2}D_p^{-\frac{1}{2}\beta_2}\leq \rho D_p^{\frac{1}{2}}\leq C^{1/2}D_p^{-\frac{1}{2}\beta_2}$$
and we obtain that 
$$
\left|z'_{2N+p}(\omega)\right|\sim \left|z'_{2N}(\omega)\right|2apD_p\sim 2aD_{N}^{1-\eta}D_p^{\frac{1}{2}-\frac{1}{2}\beta_2}.
$$
We now choose $\eta=\frac{3}{2}+\epsilon$. This means that 
$$
\left|z'_{2N+p}(\omega)\right|\geq 2aD_{N}^{-\frac{1}{2}-\epsilon}D_p^{\frac{1}{2}-\frac{1}{2}\beta_2}=2aD_{N}^{-\frac{1}{2}-\epsilon}D_N^{\left(\frac{1}{2}-\frac{1}{2}\beta_2\right)\frac{2\left(\frac{1}{2}+\epsilon\right)}{1+\beta_2}}.
$$
If $\epsilon=\frac{\beta_2}{2}$ we obtain that 
$2aD_{N}^{-\frac{\beta_2}{2}-\frac{\beta_2}{2}}=2aD_{N}^{-{\beta_2}}$.

We then follow the segment until the next return $2N+p+\ell$ and 
$$
\left|z'_{2N+p+\ell}(\omega)\right|\geq \text{const}\,D_{N}^{-{\beta_2}}.
$$
Since $D_{N}\geq e^{\kappa N}$, we obtain
$$
\left|z'_{2N+p+\ell}(\omega)\right|\geq \text{const}\, e^{-\kappa \beta_2N}
$$
and the free period satisfies $\ell\leq \beta_2\kappa\kappa_1^{-1}N$, where $\kappa_1$ is the Lyapunov exponent associated to the dynamics outside of $(-\delta,\delta)$. Moreover, the time $2N+p+\ell$ is less than or equal to $3N$. We can now relax the condition of the basic assumption, see Subsection \ref{ss:twod} and apply the machinery to a subinterval $\omega'\subset\omega$ which is chosen so that 
 $$
 \left|z'_{2N+p+\ell}(\omega')\right|\geq\frac{1}{4} \left|z'_{2N+p+\ell}(\omega)\right|.
 $$
As a consequence 
$$
\left|z'_{2N+p+\ell}(\omega')\right|\geq  \text{const'}\,e^{-\kappa \beta_2N}.
$$
The corresponding bound period for a return time to a position at horizontal distance $e^{-r'}$ with $r'\leq \beta_2 N$ has length smaller than or equal to $4\beta_2N<N$. In particular, we can use that the induction is valid up to time $N$ and we can repeat the argument for 
$
\left|z'_{2N+p+\ell}(\omega')\right|.
$
At the expiration time of the new bound period $p_1$, $
\left|z'_{2N+p+\ell+p_1}(\omega')\right| $ satisfies 
$$
\left|z'_{2N+p+\ell+p_1}(\omega')\right|\geq  \text{const}\,e^{r'(1-3\beta)}\left|z'_{2N+p+\ell}(\omega')\right|,
$$ see \eqref{eq:largedeviation}.
After a finite number of steps $s$, at time $n_s$ and for a parameters interval $\omega^{(s)}$, we have 
$$
\left|z_{n_s}\left(\omega^{(s)}\right)\right|\geq \frac{1}{10}.
$$
We are then in an escape situation and the argument in Section \ref{sink} applies.
\begin{figure}[h]
\centering
\includegraphics[width=0.9\textwidth]{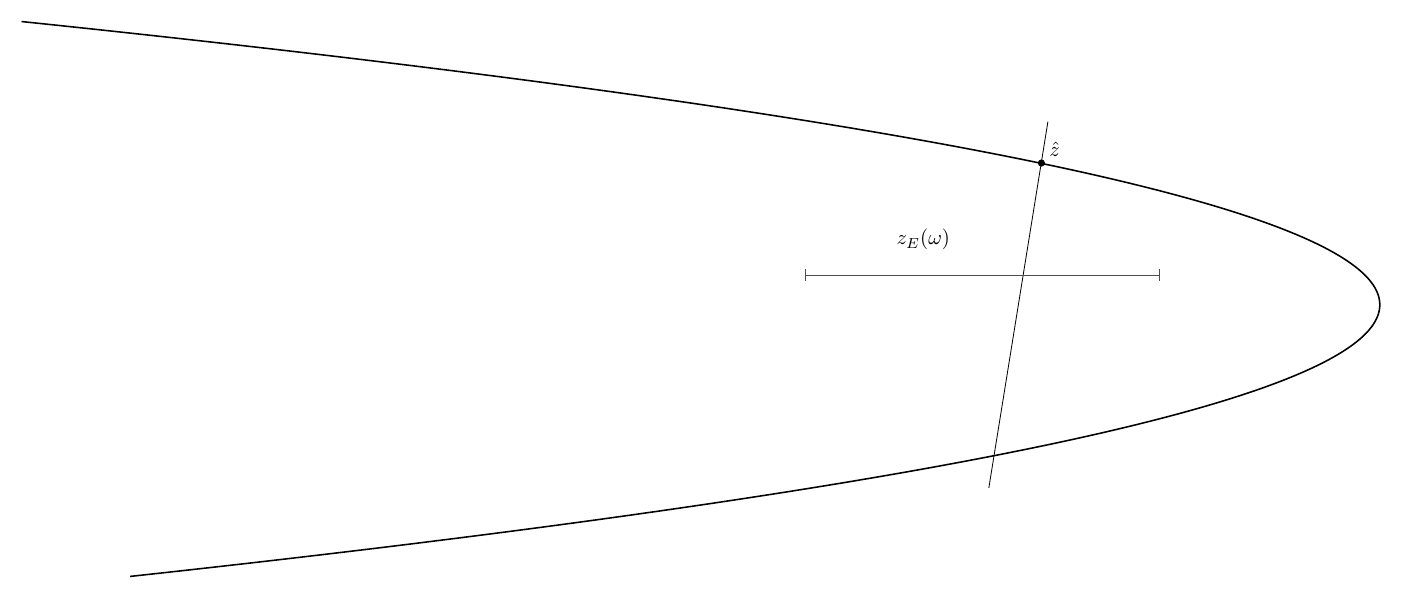}
\caption{Long escape situation for the second critical point}
\label{Fig4New}
\end{figure}
\section{Construction of a tangency}\label{sec:tangency} 
We aim to construct a non-degenerate quadratic tangency at the long
escape time $\tilde N$. We consider a parameter interval
$\tilde{\omega}$. For each $a$ in $\tilde{\omega}$ there is a critical
point $\tilde{z}_0$ and a fixed point $\hat{z}_a$
For each fixed $a\in\tilde{\omega}$ a segment $\gamma_{u}^a\subset
W^u(\hat{z}_a)$ which contains $\tilde{z}_a$ and $\hat{z}_a$.
 We aim to prove that $F_{a}^{\tilde N} (\gamma_{u}^a)$ has very high curvature near
$F_{a}^{\tilde N}(\tilde z_0(a))$. It is advantageous to study the
curvature of $F_a^{\tilde N-1}(\tilde\gamma_{u}^a)$ where
$\tilde\gamma_{u}^a$ is the curve  $F_a(\gamma_{u}^a)$ which is located close to the critical value.

We decompose the tangent vector $\tau(u)$ along $\tilde\gamma_u^a$ as 
$$\tau(u)=A(u)E_{\tilde N-1}(u)+B(u)W_{\tilde N-1}(u)$$ where 
$$\left\{\begin{matrix}
E_{\tilde N-1}(u)=e_{\tilde N-1}(\zeta_1(u))\\
W_{\tilde N-1}(u)=DF_a^{\tilde
	N-1}(\zeta_1(u))\left(\begin{matrix}
1\\0
\end{matrix}\right)
\end{matrix}\right.$$
and $\zeta_1(u)$ is the parametrization of $\tilde\gamma_{u}^a$ by arclength.

We have 
\begin{equation}\label{main}
\zeta_{\tilde N}(\rho)-\tilde z_{\tilde N}=\int_0^{\rho}\left(A(u)E_{\tilde N-1}(u)+B(u)W_{\tilde N-1}(u)\right)du.
\end{equation}
Here $\zeta_1=\zeta_1(\rho)$ is an arbitrary point on $\tilde\gamma_a^u$ at arclength $\rho$ from $\tilde {z}_1(a)$ and $\zeta_{\tilde N}(\rho)=F_a^{\tilde N-1}\left(\zeta_1(\rho)\right)$.

Differentiating \eqref{main} twice, we get 
\begin{equation}\label{main1}
\zeta'_{\tilde N}(\rho)=A(\rho)E_{\tilde N-1}(\rho)+B(\rho)W_{\tilde N-1}(\rho)
\end{equation}
and 
\begin{equation}\label{main2}
\zeta''_{\tilde N}(\rho)=A'(\rho)E_{\tilde N-1}(\rho)+A(\rho)E'_{\tilde N-1}(\rho)+B'(\rho)W_{\tilde N-1}(\rho)+B(\rho)W'_{\tilde N-1}(\rho)
\end{equation}
\begin{lem}\label{wprime} For all $\rho>0$
$$\left|W'_{\tilde N-1}(\rho)\right|\leq 25^{\tilde N-1}.$$ 
\end{lem}
\begin{proof}
Observe that 
$$
W_{\tilde N-1}(\rho)=DF(x_{\tilde N-1},y_{\tilde N-1})\dots DF(x_{1},y_{1})\begin{brsm}1\\0\end{brsm}.
$$
By differentiating with respect to $\rho$ and taking the matrix norm, one gets,
$$
\left|W'_{\tilde N-1}(\rho)\right|=\sum_{i}\left(\prod_{j\neq i}\left\|DF(x_{j},y_{j})\right\|\right)\left\|P_i\right\|
$$
where $$P_i=\frac{d}{d\rho}\left[\begin{matrix}
-2ax_i+\partial_x\psi_1 & \partial_y\psi_1\\
\partial_x\psi_2& \partial_y\psi_2
\end{matrix}\right].$$ 
Since the $C^2$ norms of $\psi_1$
and  $\psi_2$ have the bound   $Cb^{t/2}$,  see \cite{MoraViana},
Section 7A, we get 
\begin{eqnarray*}
\left|W'_{\tilde N-1}(\rho)\right|\leq \sum_{i}\left[\left(\frac{9}{2}\right)^{\tilde N-2}\cdot 3\left(\frac{9}{2}\right)^{i}\right]\leq 25^{\tilde N}
\end{eqnarray*}
where we used that $\left\|W_i\right\|<\left(\frac{9}{2}\right)^{i}$ (since $\left\|DF\right\|<\frac{9}{2}$).
%
%
\end{proof}
\begin{prop}\label{bigcurvature}
Let $ |\rho_0|=\frac{\left|E_{\tilde N}(0)\right|}{\left\|W_{\tilde
      N}(0)\right\|} $, then for a suitably chosen $\rho'_0$, and  for all $\rho'_0\leq|\rho|\leq\rho_0$,  the curvature of $\zeta_{\tilde N}(\rho)$, $\kappa\left(\zeta_{\tilde N}(\rho)\right)$ satisfies the following:
$$
{2C_1}\frac{|W_{\tilde N}(\rho)|}{|E_{\tilde N}(\rho)|^2}
\geq\kappa\left(\zeta_{\tilde N}(\rho)\right)\geq \frac{C_1}{2}\frac{|W_{\tilde N}(\rho)|}{|E_{\tilde N}(\rho)|^2}$$
with $2\leq C_1\leq 4$.
\end{prop}
\begin{rem}
Observe that the numbers $\frac{1}{2}$ and 2 appearing in the
curvature estimates above can be chosen arbitrarily close to 1, if $b$ is sufficiently small. 
\end{rem}
\begin{proof}
Recall that 
$$\kappa(\rho)=\frac{\left|\zeta'_{\tilde N}(\rho)\times\zeta''_{\tilde N}(\rho) \right|}{\left|\zeta'_{\tilde N}(\rho)\right|^3}.$$ We start by computing $\zeta'_{\tilde N}(\rho)\times\zeta''_{\tilde N}(\rho)$. We get 
\begin{eqnarray*}
\zeta'_{\tilde N}(\rho)\times\zeta''_{\tilde N}(\rho)&=&A(\rho)A'(\rho)E_{\tilde N-1}(\rho)\times E_{\tilde N-1}(\rho)+A(\rho)^2E_{\tilde N-1}(\rho)\times E'_{\tilde N-1}(\rho)\\
&+&A(\rho)B'(\rho)E_{\tilde N-1}(\rho)\times W_{\tilde N-1}(\rho)+A(\rho)B(\rho)E_{\tilde N-1}(\rho)\times W'_{\tilde N-1}(\rho)\\
&+&A'(\rho)B(\rho)W_{\tilde N-1}(\rho)\times E_{\tilde N-1}(\rho)+A(\rho)B(\rho)W_{\tilde N-1}(\rho)\times E'_{\tilde N-1}(\rho)\\
&+&B(\rho)B'(\rho)W_{\tilde N-1}(\rho)\times W_{\tilde N-1}(\rho)+B(\rho)^2W_{\tilde N-1}(\rho)\times W'_{\tilde N-1}(\rho)
\end{eqnarray*}
and since $E_{\tilde N-1}(\rho)\times E_{\tilde N-1}(\rho)=W_{\tilde N-1}(\rho)\times W_{\tilde N-1}(\rho)=0$
\begin{eqnarray*}
\zeta'_{\tilde N}(\rho)\times\zeta''_{\tilde N}(\rho)&=&
\left(A(\rho)B'(\rho)-A'(\rho)B(\rho)\right)E_{\tilde N-1}(\rho)\times W_{\tilde N-1}(\rho)\\
&+& A(\rho)^2E_{\tilde N-1}(\rho)\times E'_{\tilde N-1}(\rho)+B(\rho)^2W_{\tilde N-1}(\rho)\times W'_{\tilde N-1}(\rho)\\
&+&A(\rho)B(\rho)E_{\tilde N-1}(\rho)\times W'_{\tilde N-1}(\rho)+A(\rho)B(\rho)W_{\tilde N}(\rho)\times E'_{\tilde N-1}(\rho).
\end{eqnarray*}
In \cite{BC2}, Section 7.5 there are estimates of $A$, $A'$, $B$ and $B'$ in the classical H\'enon case. We have new similar estimates in the H\'enon-like case as follows:
\paragraph{Claim.} There are constants $\gamma_1$ and $C_1=2a/\gamma_0^2$
such that 
\begin{eqnarray}\label{A-B-est}
\frac{2}{\gamma_0^2}\rho\leq B(\rho)&\leq \frac{4}{\gamma_0^2}\rho,\\
B'(\rho)&=&\frac{2a}{\gamma_0^2}\frac{dy}{d\rho}+O(b)=C_1+O(b),\\
A(\rho)&=&1+O(\rho^2),\\
A'(\rho)&=&O(\rho).
\end{eqnarray}

We prove now the previous Claim. Observe that for $(x,y)$ close to the critical
value $\tilde{z}_1$  
$$
Df_a=\left(\begin{matrix}
-2ax+\psi_{1x}&\psi_{1y}\\
\psi_{2x}&\psi_{2y}
\end{matrix}\right)=\left(\begin{matrix}
-2ax+\alpha_1&\beta_1\\
\gamma_1&\delta_1
\end{matrix}\right)
$$
and 
$$
Df_a^{-1}=\frac{1}{\text{det}Df_a}\left(\begin{matrix}
\delta_1&-\beta_1\\
-\gamma_1&-2ax+\alpha_1
\end{matrix}\right).
$$
It follows that 
$$
(Df_a^{-1})^2\left(\begin{matrix}.
0\\
1
\end{matrix}\right)=\frac{1}{(\text{det}Df_a)^2}\left(\begin{matrix}
-\delta\beta_1-\beta_1(-2ax+\alpha_1)\\
\beta_1\gamma_1+(-2ax+\alpha_1)^2
\end{matrix}\right)
$$
 
This means that the most contractive direction for $(x,y)$ close to $\tilde{z}_1(a)$ has slope
$$
s\approx\frac{2ax-\alpha_1}{\beta_1}.
$$
By the construction of the local stable manifold, see \cite{BC2} pp. 110-111, it follows that there is a temporary stable foliation with slope of $\approx{2ax_1}/{\beta}$ with $x_1\approx 1$.

We claim that the image of the leg of the unstable manifold near the critical value is an approximate parabola. The unstable direction at $\hat{z}$ is given by the unstable direction of the fixed point located approximately at $\left(\begin{matrix}
1/2\\0
\end{matrix}\right)$.
The slope of $W^u$ near $\tilde z_0$ is given by 
$$
\left(\begin{matrix}
-2ax_{-1}+\alpha_{-1}&\beta_{-1}\\
\gamma_{-1}&\delta_{-1}
\end{matrix}\right)\left(\begin{matrix}
-2ax_{-2}+\alpha_{-2}&\beta_{-2}\\
\gamma_{-2}&\delta_{-2}
\end{matrix}\right)\dots\left(\begin{matrix}
-2ax_{-k}+\alpha_{-k}&\beta_{-k}\\
\gamma_{-k}&\delta_{-k}
\end{matrix}\right)v_0
$$
where $v_0$ is the slope of $W^u(\hat z)$ which is essentially horizontal. Observe that $x_{-1}$ is approximately given by $1-ax_{-1}^2=0$. The slope of $W^u$ near $\tilde z_0$ is approximately given by ${\gamma_{-1}}/{-2ax_{-1}}$.

By approximating the unstable manifold by a straight line
$$
\begin{cases}
x=x_0+t\\
y=y_0+kt
\end{cases}
$$
where $k={\gamma_{-1}}/{-2ax_{-1}}$ we have that the image of $W^u$ is
the curve $(x_1(t),y_1(t)$ with the derivative
$$
\begin{cases}
x_1'(t)=-2a(t+x_0) +\alpha_0+k\beta_0\\
y_1'(t)=\gamma_0+k\delta_0
\end{cases}
$$
so
The curvature is then given by
$\kappa(t)=|\gamma_1'(t)\times\gamma_2''(t)|/|\gamma_1'(t)|^3\approx 2a/\gamma_0^2$.

This means that, in a suitable almost orthogonal coordinate system  $(\eta_1,\xi_1)$, one can use a version of Hadamard's lemma, see Lemma 8.7 in \cite{BC2} to get that the image parabola looks approximately as 
$$
\xi_1=1-a\left(\frac{\eta_1}{\gamma_0}\right)^2.
$$

For convenience of the reader we recall here {\it Lemma 8.7,\cite{BC2}} which we just used. Let $f\in C^2(A,A+\ell)$ and suppose that
$$
|f(a)|\leq M_0,\qquad |f''(a)|\leq M_2.
$$
Then if
$$
4M_0 < \ell^2
$$
it hold that
$$
|f'(a)|\leq \sqrt{M_0}(1+M_2).
$$
This completes the proof of the Claim.
\bigskip

The following estimates hold.
\begin{eqnarray*}
\left|\left(A(\rho)B'(\rho)-A'(\rho)B(\rho)\right)E_{\tilde N-1}(\rho)\times W_{\tilde N-1}(\rho)\right|&\geq&\frac{3}{4}C_1\left|E_{\tilde N-1}(\rho)\times W_{\tilde N-1}(\rho)\right|\\
&\geq&\frac{C_1}{2}\left|E_{\tilde N-1}(\rho)\right|\left| W_{\tilde N-1}(\rho)\right|,
\end{eqnarray*}
where we used the fact that the angle between $W_{\tilde N-1}$ and $E_{\tilde N-1}$ is very close to $\pi/2$l, see formula $(9)$, Section $6$ in \cite{MoraViana}. 
By Lemma $6.8$ in \cite{MoraViana}, we get
\begin{eqnarray*}
\left|E_{\tilde N-1}(\rho)\times E'_{\tilde N-1}(\rho)\right|&\leq&\left|E_{\tilde N-1}(\rho)\right| \left| E'_{\tilde N-1}(\rho)\right| \\
&\leq&\left|E_{\tilde N-1}(\rho)\right| \left(K_1 b\right)^{\tilde N-4}
\end{eqnarray*}
with $K_1>0$.
By Lemma \ref{wprime} we have
\begin{eqnarray*}
\left|B(\rho)^2 W_{\tilde N-1}(\rho)\times W'_{\tilde N-1}(\rho)\right|&\leq&\left|W_{\tilde N-1}(\rho)\right|\frac{4^2}{\gamma_0^2}\rho^2 25^{\tilde N-1}
\\
&\leq&\left|W_{\tilde N-1}(\rho)\right|\left|E_{\tilde
       N-1}(\rho)\right|\cdot\frac{16}{\gamma_0^2}\left(\frac{\left|E_{\tilde
       N-1}(\rho_0)\right|^2}{\left|W_{\tilde N-1}(\rho_0)\right|^2 |E_{\tilde{N}-1}(\rho)|}25^{\tilde N-1}\right)\\
&\leq&\frac{1}{100}\left|W_{\tilde N-1}(\rho)\right|\left|E_{\tilde N-1}(\rho)\right|,
\end{eqnarray*}.
We have used the distorsion estimate for $W$-vectors, see \cite{MoraViana},
Lemma 10.2. to conclude that $W_{\tilde{N}-1}(\rho)$ and
$W_{\tilde{N}-1}(\rho_0)$ are comparable, the estimate that
$|W_{\tilde{N}-1}(\rho)|\cdot |E_{\tilde{N}-1}(\rho)|\approx
b^{t(\tilde{N}-1)} $ and the estimate $\left|E_{\tilde
    N-1}(\rho)\right|<\left({Kb}/{\kappa}\right)^{\tilde N-1}$

By Lemma $6.8$ in \cite{MoraViana}, 
\begin{eqnarray*}
\left|A(\rho)B(\rho) W_{\tilde N-1}(\rho)\times E'_{\tilde N-1}(\rho)\right|&\leq& \frac{8}{\gamma_0^2} |\rho| \left| W_{\tilde N-1}(\rho)\right|\left(K_1 b\right)^{\tilde N-4}
\\
&\leq& \frac{8}{\gamma_0^2}\frac{\left| E_{\tilde N-1}(\rho_0)\right|}{\left| W_{\tilde N-1}(\rho_0)\right|}\left| W_{\tilde N-1}(\rho_0)\right|\left(K_1 b\right)^{\tilde N-2}\\
&\leq&\frac{1}{100} \left| E_{\tilde N-1}(\rho_0)\right|\left| W_{\tilde N-1}(\rho_0)\right|.
\end{eqnarray*}
By Lemma \ref{wprime} we have
\begin{eqnarray*}
\left|A(\rho) B(\rho) E_{\tilde N-1}(\rho)\times W'_{\tilde N-1}(\rho)\right|&\leq&\frac{8}{\gamma_0^2} |\rho| \left|E_{\tilde N-1}(\rho)\right| 25^{\tilde N-1}
\\
&\leq&\frac{8}{\gamma_0^2}\left|W_{\tilde N-1}(\rho_0)\right|\left|E_{\tilde N-1}(\rho_0)\right|\frac{\left|E_{\tilde N-1}(\rho_0)\right|}{\left|W_{\tilde N-1}(\rho_0)\right|^2}  25^{\tilde N-1}\\
&\leq&\frac{1}{100}\left|W_{\tilde N-1}(\rho_0)\right|\left|E_{\tilde N-1}(\rho_0)\right|,
\end{eqnarray*}
where we used that $|\rho|^2\leq |\rho_0|^2={\left|E_{\tilde
      N}(\rho_0)\right|^2}/{\left|W_{\tilde N}(\rho_0)\right|^2} $
and $\left|E_{\tilde N-1}(\rho)\right|<\left({Kb}/{\kappa}\right)^{\tilde N-1}$, $K,\kappa>0$, see formula $(5)$ of Section $6$ in \cite{MoraViana}.

The term that dominates is $A(\rho)B'(\rho)E_{\tilde{N-1}}(\rho)\times
W_{\tilde{N-1}}(\rho)$. The proof of the lemma is concluded by combining the previous five estimates. 
\end{proof}
\subsection{Quadratic Tangency}
We prove that in a long escape situation a quadratic tangency appears.
\begin{prop}\label{prop:tangency}
Let $z_E(\omega)$ be a curve segment  of critical values in an escape
situation that intersect {$\gamma^s$}, the leg of $W^s(\hat{z})$
pointing downwards.
Then there exists a unique $a_0\in\omega$ such that the tangency between {$\gamma^s_{a_0}$} and {$\gamma^u_{a_0}$} is quadratic. 
\end{prop}
\begin{rem}\label{rem:curv}
Actually, the curvature of {$\gamma^s_{a_0}$} is close to zero while the curvature of $\gamma^u_{a_0}$ is close to its maximal which is $2\frac{|W_N|}{|E_N|^2}$ within a factor close to $1$.
\end{rem}
\begin{proof}
By Proposition \ref{bigcurvature}, the $\rho$ which makes the slope
equal to $-{C}/{\sqrt b}$ is
roughly $$\rho=-\frac{|E_N|}{2C|W_N|}\sqrt b.$$ Observe that this
$\rho$ satisfies the estimate $|\rho|\geq\rho_0$, so we avoid the
exact tip of the parabola like image of the unstable manifold. We use
the bounds in Proposition \ref{bigcurvature}for the curvature and the
angle between $E_{\tilde{N}}(0)$ and $W_{\tilde{N}}(0)$ is
$\frac{\pi}{2}$. Using that $||W_{\tilde{N}}||\leq 25^{\tilde{N}}$
(Lemma \ref{wprime}) and $||DE_{\tilde{N}}||\leq C\sqrt{b}$
(\cite{MoraViana}, Lemma 6.6), the statement follows.  

\end{proof}
\begin{figure}[h]
\centering
\includegraphics[width=0.9\textwidth]{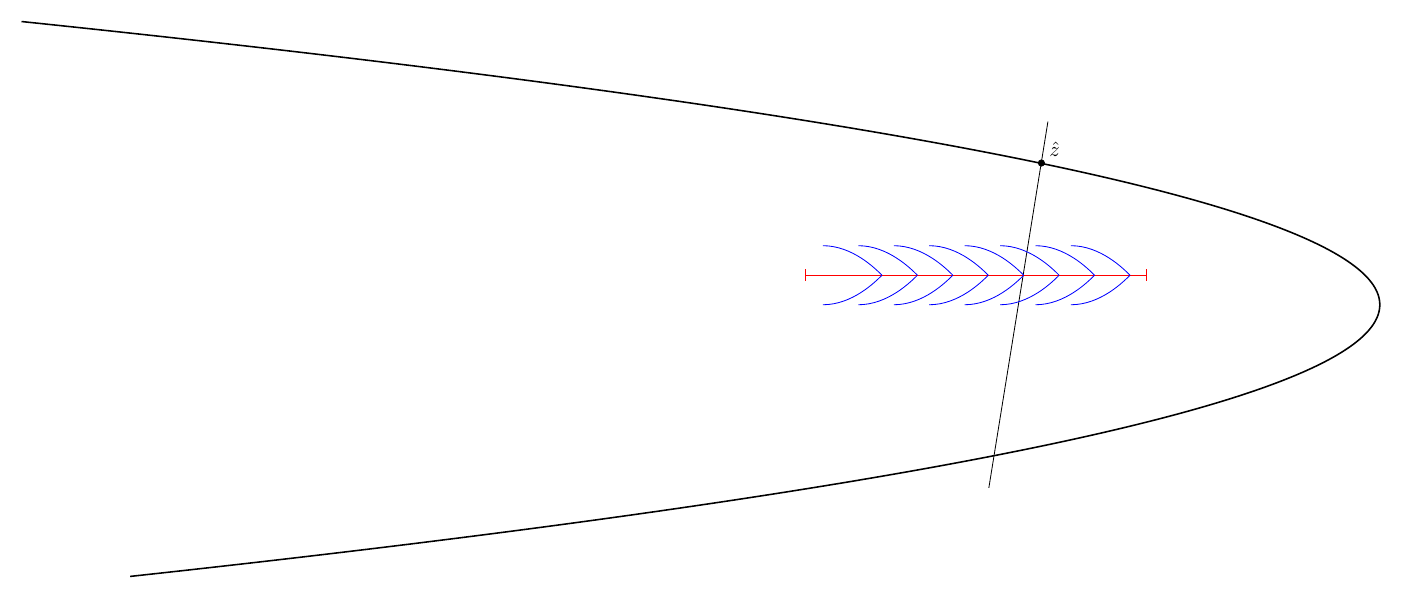}
\caption{Quadratic tangency}
\label{Fig5New}
\end{figure}
\section{Proof of theorems \ref{A}, \ref{B} and \ref{C}}
The proof of theorems \ref{A}, \ref{B} and \ref{C} is done by
induction. From sections \ref{sink} and \ref{sec:tangency} we selected
maps with a sink and a new tangency. We reapply now Section \ref{sink}
to get a second sink and Section \ref{sec:tangency} to get a new
tangency. One could stop this process after $k$ steps. At this moment
one would have $k$ sinks and a new tangency. This tangency will then
be used to create a strange attractor using \cite{MoraViana} and give
the proof of Theorem \ref{A}.
Alternatively, one could continue the process infinitely many times to
get infinitely many sinks. This leads to the proof of Theorem \ref{B}.
The inductive procedure is formulated in the next proposition.
\begin{prop}\label{mainprop}
There exists $K>0$ such that, for all $k=0,1,\dots , K$, there are
parameters intervals $\omega_k$ with $\omega_k\subset\omega_{k-1}$, so
that, for all $a\in\omega_k$, there is a $\mathcal C^2(b)$ curve
$\gamma_k(a)\subset W^u(\hat z)$ with
$z_k(a)\in\gamma_k(a)$. Moreover, for all $k=0,1,\dots, K$ there are
regions $\mathcal D_{N_k}(a)$ with $\mathcal D_{N_j}(a)\cap \mathcal
D_{N_i}(a)=\emptyset$ for all $i\neq j$ such that $\mathcal
D_{N_k}(a)$ is bounded by $\gamma_k(a)$ and parabolic leaves of
$W^s_{\text{loc}}$ and it contains a unique sink.
\end{prop}

\begin{proof} We proceed by induction and the case of one sink appears
  in Section \ref{sink}.
Assume that we have already constructed $k$ sinks and that a
parameter interval $\omega^{(k)}$ corresponding to the critical point
$z^{(k+1)}_0$ is in escape situation and intersects $W^s(\hat{z})$. We
now have an unfolding of a homoclinic tangency as in Palis-Takens
\cite{PalisTakens} and \cite{MoraViana}. We can then do the
renormalization procedure associated to this unfolding as in these
papers and we obtain a new renormalized H\'enon-like family. This
allows us to create a new sink as in Section 3, and we obtain also a
new escape situation following the argument in Section \ref{sec:capturing}.
\end{proof}

{\it Proof of Theorem \ref{A}.} The proof is a small modification of that of
Proposition \ref{mainprop}. The only difference is that, at the time $k$, instead of construct a new sink one can create a strange attractor as in \cite{MoraViana} at the homoclinic unfolding.\demo

\medskip
{\it Proof of Theorem \ref{B}.} The proof is a minor modification of that of {Theorem \ref{A}}. The only difference is that instead of
switching to construction of a strange attractor after $k$ steps, we
continue to construct more and more sinks. We obviously obtain
Newhouse parameters in the limit. 
Note that the renormalizations take parameters of a specific H\'enon-like family linearly to new renormalized parameters of the corresponding H\'enon-like family. For each renormalization of order $k$, we get a set $A'_k$ of parameters in the renormalized H\'enon-like family of maps with $k$ sinks. We denote by $A_k$ the pullback of $A'_k$ containing parameters of the original H\'enon-like family. Consider now a non-empty closed subset of $A_k$, $B_k$ and denote by $B'_k$ the push-forward of $B_k$. We do at this point, another renormalization and we get a sequence of inclusions
$$
A_1\supset B_1\supset\dots\supset A_k\supset B_k\supset A_{k+1}\supset\dots.
$$ The intersection $$
\bigcap_{k=1}^{\infty}A_k
$$ is then non-empty and so is then the set of maps with infinitely many sinks.


\demo

{\it Proof of Theorem \ref{C}.} This result is a direct consequence of
Theorem \ref{A} and Theorem \ref{B}, since the H\'enon family is a
special example of a H\'enon-like family.\demo

\section{Construction of two coexisting strange attractors}
In this section we prove the existence of two strange attractors for a
parameter set of positive Lebesgue measure within the classical
H\'enon family.

We first outline the proof.
The idea is to find parameters with two coexisting homoclinic
tangencies. To do this we consider two very close critical points
which are in escape situation simultaneously. We must chose them very
carefully so that their images are at suitable distance at the escape
situation. To do this we have to chose carefully their initial distance and the time they spend in the hyperbolic region
outside of $(-\delta,\delta)$. We will create one true tangency for
the first critical point at the point $a_0$ and then we create a
tangency for a second  critical point but for a different parameter
value $a_0'$. Both critical points will have associated parameter sets
of positive one-dimensional Lebesgue measure with different strange
attractors. These parameter sets will intersect if the parameters with
the respective strange attractor are abundant at the respective points.  

\medskip
We return to the construction of the first critical point $z_0$ and
the corresponding long escape situation of Section
\ref{sec:long_escape_situation}. We fix $b<b_0$ and by  Lemma
\ref{lem:long_escape_sit} we see that there is a subinterval
$\tilde{\omega}_0$ such that $z_E(\tilde{\omega}_0)$ is in a long
escape situation.

We now construct a second critical point $\tilde{z}_0$. The construction
is similar to the corresponding one in Section \ref{sec:capturing}. The
difference is that $\tilde{z}_0$ will be chosen much closer to $z_0$
vertically than $\tilde{z}_0$ is to $z_0$ and its distance can be
chosen exponentially well spaced, see \eqref{eq:spacingofwuleaves}. 
From Lemma \ref{lem:newcapture}, choose $j$ and the corresponding $\hat{z}_0$ so that $j$ is the
minimal integer so that  for all
$a\in\tilde{\omega}_0$ at time $E$, $\tilde{z}_E$ is still bound to
$z_E$.




  \medskip
\subsection{Proof of Theorem \ref{D}}

We start with the construction of the first critical point $z_0$ and follow it until the first escape situation which appears at time $T_0$. 

Close to $z_0$ we have a number of critical points which are inter-spaced as follows.

\begin{prop}\label{critpts} Close to the critical point $z_0$ we 
have a sequence of critical points $\tilde{z}_0^{(j)}$ so that 
with $\lambda_1$ denoting the unstable eigenvalue, 
$$
|z_0-\tilde{z}_0^{(j)}|\leq C\left(\frac{b}{|\lambda_1|}\right)^j
$$
and
$$
C'\frac{b}{|\lambda_1|}|\tilde{z}_0^{(j)}-\tilde{z}_0^{(j-1)}|\leq |\tilde{z}_0^{(j+1)}-\tilde{z}_0^{(j)}|\leq C''\frac{b}{|\lambda_1|}|\tilde{z}_0^{(j)}-\tilde{z}_0^{(j-1)}|.
$$
\end{prop}
\begin{proof} The stable manifold at the fixed point $\hat{z}$
intersects the first leg of the unstable manifold in a homoclinic point $z_h$. Segments around $z_h$ are captured towards $\hat{z}$
and by the Lambda lemma these segments are accumulated on the unstable manifold, in particular at $z_0$. The behaviour is dominated by the behaviour at the fixed point $\hat{z}$ and is dominated by the stable eigenvalue $\lambda_2=b/\lambda_1$.
\end{proof}

We want to chose a $j$ so that $\tilde{z}_0^{(j)}$ is a suitable distance to $z_0$ so that at an escape time 
$E$ booth the point $z_0$ and $\tilde{z}_0^{(j)}$ escapes.

Consider a subinterval $\omega$ in the parameter space that escapes at time $T_0$. We can accomplish that $|z_0(\omega)| \sim\frac{1}{10}$. Now chose a second critical point  $\tilde{z}_0^{(j)}(a)$ for $a\in\omega$. Denote the initial distance between $z_0(a)$ and 
$\tilde{z}_0^{(j)}(a)$ for a fixed $a\in\omega$ by $d_j$. By Proposition \ref{critpts} 
it follows that 
$$
d_j=\mathcal{O}\left(\frac{b}{|\lambda_1|}\right)^j.
$$
It follows that at an escape  time $T_0$

$$
|z_{T_0}(a)-\tilde{z}_{T_0}^{(j)}(a)|\sim\left(\frac{b}{|\lambda_1|}\right)^j||W_{T_0}(a)||,
$$
where $\sim$ denotes that the quotients of the two sides are bounded above and below
with fixed constants.

We want to accomplish that at a suitable time $T_0+L$ there are simultaneous intersections of  $z_{T_0+L}(\omega')$  and  $\tilde{z}_{T_0}^{(j)}(\omega')$ with different legs of the stable manifold $W^s(\hat{z})$ for all $a\in{\omega'}$.
To achieve this we study the  distribution of the vertical segments of $W^s$. We consider the tent map
$$
y\mapsto 1- 2|y|
$$
which is conjugate to the full quadratic map
$$
1-2x^2
$$
The preimages of the fixed point $y=\frac{1}{3}$ of  the tent map are located in
$$
y_{k,\nu}=\frac{k}{3\cdot 2^\nu},\qquad
k=-3\cdot 2^\nu-1,\dots,3\cdot 2^\nu+1.
$$
These points correspond to $x_{k,\nu}=\sin \frac{\pi}{2} y_{k,\nu}$ and
by continuity this is approximately true for all $a$ close to 2. To
each $x_{\nu,k}$ there corresponds an almost vertical branch $\gamma_{k,\nu}$ of $W^s$
 and, by chosing $j$ and $L$ appropriately, we
will have the situation that $\tilde{z}_0^{(j)}(a_0)$ is located
between  $\gamma_{k,\nu}$ and $\gamma_{k+1,\nu}$ for suitable chosen
$k$ and $\nu$. This follows by the following argument. Suppose that
for a given orbit $z_{E+j}(\omega')$, $j\geq 0$, moves outside
$(-\delta,\delta)\times{\mathbb R}$. Note that $||Df|| \leq 5$. By
Lemma 4.5 in \cite{BC2} it follows that the slope $s_{T+j}$ of
$w_{T_0+j}$ satisfies $|s_{T_0+j}|\leq b/\delta$ if we restrict to a suitable
parameter interval $\omega'$. 

After restricting $\omega'$ further if necessary we can obtain that
for some $j=L$, say  $z_{T_0+j}(\omega')$, stretches across one
$\gamma_{k,\nu}$ and the stable leg of $W^s(\hat{z})$. Denote the
intersection points by $a_0'$  and $a_0''$ respectively.

We will need information about the local behavior at the image.
It follows from Propostion \ref{bigcurvature} that there are  points of tangencies for parameters
$\tilde{a}_0'$, $\tilde{a}_0'']$ close to  $a_0'$   respectively.





Let us consider the homoclinic tangencies that appears in Proposition
\ref{prop:tangency} for the parameters $as=\tilde{a}'_0$ and
 $as=\tilde{a}''_0$ at time $T_0+L$

Suppose that the common tangency occurs for a parameter $a_0$. We
consider the normalization argument in \cite{PalisTakens}. 

The curvature is given, by
$$
Q_1=\frac{|W_N^1(\rho)|}{|E_N^1|^2\left({1+\left(\frac{|B(\rho)||W_N^1|}{|E_N^1|^2}\right)^2}\right)^{3/2}}
$$

The maps $\varphi_\mu^N$ are written in coordinates
$$
(1+x,y)\mapsto (0,1)+(H_1(\mu,x,y),H_2(\mu,x,y))
$$

with

 \begin{align*}
 H_1(\mu,x,y)&=v\cdot x^2+\mu+w y+\tilde{H}_1(\mu,x,y)\\
 H_2(\mu,x,y)&=u\cdot y+\tilde{H}_2(\mu,x,y).
 \end{align*} 

They define $N$ dependent on
reparametrization of the parameter $\mu$ and a $\mu$-dependent change of
coordinates 
renormalizations. The parameter renormalization is given by
$$
\overline{\mu}=\sigma^{2N}\cdot \mu+w\cdot\kappa^N\cdot\sigma^{2N}-\sigma^N.
$$

In the renormalized coordinates the parameter interval is
$[\overline{\mu}'_0,2]$. We write
\begin{equation}\label{decomp'}
[\overline{\mu}'_0,2]=\bigcup_{r\geq r_0} J'_r=J'_{r,\ell}
\end{equation}

with $J'_{r,\ell}$ disjoint and $|J'_{r,\ell}|=\frac{1}{r^2}|J_r'|$.

We do a similar decompostion of  $[\overline{\mu}''_0,2]$.

\begin{equation}\label{decomp''}
[\overline{\mu}''_0,2]=\bigcup_{r\geq r_0} J''_r=J''_{r,\ell}
\end{equation}

To each of the decompostions \eqref{decomp'} and \eqref{decomp''}
we get two decompostions

\begin{equation}\label{adecomp'}
[\tilde{a}'_{0,0},\tilde{a}_0']=\bigcup_{r\geq r_0} \omega'_r=\omega'_{r,\ell},
\end{equation}

and

\begin{equation}\label{adecomp''}
[\tilde{a}''_{0,0},\tilde{a}_0'']=\bigcup_{r\geq r_0} \omega''_r=\omega''_{r,\ell}.
\end{equation}

There is a uniform distorsion bound for the parameter maps $a'\mapsto
\nu'$ and $a''\mapsto\nu''$

For each $\omega=\omega_{r,\ell}'$ we do the parameter selection to
create a strange attractor as in \cite{MoraViana}.

Let $n$ be an essential
free return time and let $E_n(z_0)$ be the set
kept at time $n$ where we take into accont the parameter deletions
because of the (BA) conditions and the large deviation estimate. 
By formula (1) in Section 12 in \cite{MoraViana} the measure of the deleted
set $\omega\setminus E_n(\omega)$ satisfies
\begin{equation}\label{eq:del-set1}
m(\omega\setminus E_n(\omega))\leq B_0 e^{-\alpha_0 n} m(\omega),
\end{equation}  
where $B_0$ and $\alpha_0$ depend on  $K,\alpha,\beta$ and $\delta$
but not on $N$ and $b$. 

We define
\begin{equation}\label{eq:del-set2}
E_n(\omega)=E_{n-1}(\omega)\setminus \left(\bigcup_{z_0}
  (\omega\setminus E_n(z_0))\right).
\end{equation}  
The number of critical points is
$$
\#C_n\leq 4 \left(\frac{K}{\rho_0}\right)^{\theta n}
$$
by Section 12 in \cite{MoraViana}.

This leads to
\begin{align*}
m(E_{n-1}\setminus E_n)&\leq 4B_0 4
                         \left(\frac{K}{\rho_0}\right)^{\theta n}e^{-\alpha_0n} m(\omega)\\
  = &4B_0\left(\left(\frac{K}{\rho_0}\right)^{\theta
      }e^{-\alpha_0}\right)^nm(\omega)\\
  &\leq 4B_0e^{-\alpha_0n/2}m(\omega)
\end{align*}
and
\begin{equation}\label{eq:deleted-par}
  m(\bigcup_{n\geq N}(E_{n-1}\setminus E_n))\leq \sum_{n\geq N}4B_0e^{-\alpha_0n/2}m(\omega). 
\end{equation}    
This means that for $N$ sufficiently large only a small proportion of $\omega$
will be deleted.

We now turn to the construction of simultanous attractors.

The first
attractor will be contructed as above and the second attractor will
be chosen corresponding to intervals $\omega''=\omega''_{r,\ell}$.

We now prove the coexistence of two attractors. To the critical point
$z_0$ there corresponds a parameter interval $\Omega_0=[{a}_{0,0}',{a}'_0]$
and to $\tilde{z}_0$ there corresponds the parameter interval
$\Omega_0''=[a_{0,0}'',a_0'']$. The subintervals $\omega'_{r',\ell'}$
and  $\omega''_{r'',\ell''}$ will have intersections for suitable
chosen $r',\ell',r'',\ell''$, even for a number of adjacent
$(r',\ell')$ and  $\omega''_{r'',\ell''}$. Because of the estimate 
\eqref{eq:deleted-par} the corresponding sets $E'_{r',\ell'}$ and  $E''_{r'',\ell''}$
will have a nonempty intersection.




We also have to verify that the two attractors are distict. This
follows since the attractors can be chosen to be arbitrarily well
localized and close to the different homoclinic tangencies. 

\medskip

We can now finish the proof of the main theorem of the section.

\medskip
{\it Proof of Theorem \ref{D}.} 
Consider a $b$ interval $(b_1,b_2)$, $b_2>b_1>0$, and $b_2$
sufficiently small. For each $b\in (b_1,b_2)$ there is a set $E_b$ of
positive Lebesgue measure so that there are two strange attractors and
the result follows by Fubinis' theorem.

\end{document}